\documentclass[11pt]{amsart}
\usepackage{amsmath,amsfonts,amsthm,amssymb,amscd, verbatim, latexsym}
\def\classification#1{\def\@class{#1}}
\classification{\null}

\DeclareFontFamily{OT1}{rsfs}{}
\DeclareFontShape{OT1}{rsfs}{n}{it}{<-> rsfs10}{}
\DeclareMathAlphabet{\mathscr}{OT1}{rsfs}{n}{it}

\DeclareMathOperator{\Ad}{Ad}

\DeclareMathOperator{\GL}{GL}

\DeclareMathOperator{\SL}{SL}

\DeclareMathOperator{\vdeg}{\overrightarrow{\text{deg}}}

\DeclareMathOperator{\SU}{SU}

\newcommand{\liea}{\mathfrak{a}}
\newcommand{\lieb}{\mathfrak{b}}
\newcommand{\lieg}{\mathfrak{g}}

\newcommand{\lieh}{\mathfrak{h}}

\newcommand{\lieu}{\mathfrak{u}}

\newcommand{\liex}{\mathfrak{x}}
\newcommand{\liegl}{\mathfrak{gl}}
\newcommand{\liev}{\mathfrak{v}}
\newcommand{\liew}{\mathfrak{w}}
\newcommand{\charley}{\rm{char}}
\newcommand{\Z}{\mathbb{Z}}

\newcommand{\Kbar}{\overline{K}}

\newtheorem{prop}{Proposition}[section]
\newtheorem{thm}{Theorem}

\newtheorem{cor}[prop]{Corollary}
\newtheorem{lem}[prop]{Lemma}

\numberwithin{equation}{section}

\begin{document}
\title{Growth in solvable subgroups of $\GL_r(\mathbb{Z}/p\mathbb{Z})$}
\author{Nick Gill}
\address{Department of Mathematics and Statistics, The Open University, Milton Keynes MK7 6AA, UK}
\author{Harald Andr\'es Helfgott}
\address{ENS-DMA, 45 rue d'Ulm, F-75230, Paris, France}

\begin{abstract}
Let $K=\Z/p\Z$ and let $A$ be a subset of $\GL_r(K)$ such that $\langle A
\rangle$ is solvable. We reduce the study of the growth of $A$ under the
group operation to the nilpotent setting. Fix a positive number $C\geq 1$; we prove that either $A$ grows (meaning $|A_3|\geq C|A|$), 
or else there are groups $U_R$ and $S$, with $U_R\unlhd S \unlhd \langle A\rangle$, such that $S/U_R$ is nilpotent, $A_k\cap S$ is large and $U_R\subseteq A_k$, where $k$ depends only on the rank $r$ of $\GL_r(K)$.

Here $A_k = \{x_1 x_2 \dotsb x_k : x_i \in A \cup A^{-1} \cup \{1\}\}$, and the implied constants depend only on the rank $r$ of $\GL_r(K)$.

When combined with recent work by Pyber and Szab\'o, the main result of this paper implies that it is possible to draw the same conclusions without supposing that $\langle A \rangle$ is solvable.



\end{abstract}

\maketitle

\section{Introduction}
Growth in abelian groups has been the focus of classical additive
combinatorics;
the topic is well-studied by now, though much remains to be known.
The study of growth in other groups by means of related techniques
is a more recent phenomenon. 

It is now understood that nilpotent groups behave, in broad terms,
 partly like abelian groups when it comes to growth; for example,
true analogues of Freiman's theorem can be proven to hold there.
 Growth in simple groups -- which is qualitatively different -- 
was studied in \cite{helfgott2}, and the techniques involved were
generalised and developed further in \cite{helfgott3}; after
further work (\cite{bgsu2}, \cite{helfgill}, \cite{dinai}
and \cite[\S 4.1]{varju}), a 
generalisation to all simple groups of bounded rank was completed in
 \cite{ps2} and \cite{bgt2}.

It remains to consider growth in solvable groups, which are in some
sense complementary to simple groups, and display, in general, behaviour
different from that of nilpotent groups. There was some work
on this in \cite{helfgott3}, but the general case remained unsolved.

The main result of this paper is the following:
\begin{thm}\label{t: main}
Let $K=\Z/p\Z$, and let $A$ be a subset of $\GL_r(K)$ such that $\langle A \rangle$ is solvable. Then, for every $C\geq 1$, either
\begin{enumerate}
 \item\label{i: one} $|A_3|\geq C|A|$, or else
\item\label{i: two} there is a unipotent subgroup $U_R$, a solvable group $S$ and an integer $k\ll_{r}1$, such that
\begin{itemize}
 \item $U_R$ and $S$ are both normal in $\langle A \rangle$, and $S/ U_R$ is nilpotent,
\item $A_k$ contains $U_R$, and
\item $|A_k\cap S| \geq C^{-O_r(1)}|A|$. 
\end{itemize}
\end{enumerate}
\end{thm}

Here we write $p$ for a prime number, and $A_k$ for $\{g_1 \dotsb
g_k : g_i\in A\cup A^{-1} \cup \{1\} \}$. For variables $x,y,z$ taking values in $\mathbb{R}$ we write $x=O_y(z)$ to mean that there is a function $f:\mathbb{R}\to\mathbb{R}$ such that $|x|\leq f(y)z$.

Note that, if (\ref{i: one}) does not hold, then 
$|A_k\cap S| \geq C^{-O_r(1)} |A|$ implies immediately that
$A$ is contained in the union of at most $C^{O_r(1)}$ left
(or right) cosets of $S$ (see Lem. \ref{lem:odt}). This is meaningful as soon as $C< |A|^{\delta_r}$, $\delta_r>0$
a constant; in other words, Thm.~\ref{t: main} is stating that $|A_3|\geq
|A|^{1+\delta}$ (for any $\delta \in (0,\delta_r)$) unless $A$
is contained in relatively few cosets of a solvable group (and obeys some additional conditions). In other words, Thm.~\ref{t: main} is
within the family of quantitatively strong results originating in
\cite{helfgott2}.


\subsection{Two extensions}\label{s: extensions}

It turns out that, with a little work, we can strengthen Thm.~\ref{t: main} twice over. The first such improvement will be proved by combining Thm.~\ref{t: main} with work of Pyber and Szab{\'o}. By mutual agreement, this result will be considered joint work with them.

\begin{thm}\label{t: main2}
Let $K=\Z/p\Z$, and let $A$ be a subset of $\GL_r(K)$. Then for every $C\geq 1$, either
\begin{enumerate}
 \item\label{i: one2} $|A_3|\geq C|A|$, or else
\item\label{i: two2} there are two subgroups $H_1\leq H_2$ in $\GL_r(K)$ and an integer $k\ll_{r}1$, such that
\begin{itemize}
 \item $H_1$ and $H_2$ are both normal in $\langle A \rangle$, and $H_2/ H_1$ is nilpotent,
\item $A_k$ contains $H_1$, and
\item $|A_k\cap H_2|\geq C^{-O_r(1)}|A|$.
\end{itemize}
\end{enumerate}
\end{thm}

To make things clear: we are able to remove the requirement that $\langle A \rangle$ is solvable, and state the result for all subsets of $\GL_r(\Z/p\Z)$ (note that, in this more general setting, we cannot conclude that $H_1$ is unipotent). In effect, Thm. \ref{t: main2} reduces the study of the growth of any set in $\GL_r(\Z/p\Z)$ to the nilpotent setting.

It is reasonable to think that a result similar to Thm. \ref{t: main2} should hold for $K$ any finite field; indeed such a result has been conjectured by Lindenstrauss and the second author \cite{taoblog}. In this more general setting, however, it is unclear whether we can find subgroups $H_1$ and $H_2$ with all of the given properties, particularly that of being normal in $\langle A \rangle$. The proof of Thm. \ref{t: main2} that we give in \S\ref{s: extension} relies on the fact that, in unipotent subgroups of $\GL_r(K),$ a subgroup chain $U_1>U_2>\cdots$ has length less than $r^2$. We cannot, of course, use this fact when $K$ is an arbitrary finite field.

The second improvement will be proved by combining Thm.~\ref{t: main2} with work of Tointon \cite{tointon}.\footnote{We thank an anonymous referee for pointing out that our results can be extended in this way, and for sketching the proof.}

\begin{thm}\label{t: main3}
Let $K=\Z/p\Z$, let $C\geq 1$ and let $A$ be a $C$-approximate subgroup of $\GL_r(K)$. Then $A$ is $\exp(C^{O_r(1)})$-controlled by a coset nilprogression of rank $C^{O_r(1)}$ and step at most $r$ that is contained in $A^{C^{O_r(1)}}$.
\end{thm}

Theorem \ref{t: main3} is proved in \S\ref{s: extension 2}, where we also
explain the terminology introduced in the statement. Thm.~\ref{t: main3}
represents the state-of-the-art for general statements concerning growth in 
$\GL_r(K)$; getting polynomial dependencies here would require proving
the Freiman-Ruzsa theorem with polynomial dependencies over abelian groups
-- and that is a difficult open problem \cite{green}. 

\subsection{Methods and structure of the paper}

Our main result is ultimately based on Prop.~\ref{prop:generous}, which is
an improved version of a result of the second author's
(\cite[Cor. 3.2]{helfgott3}). This result has sometimes been called
a ``sum-product result for group actions". This is correct in a historical
sense, in that it has its roots in the sum-product theorems of the
type in \cite{bkt, gk, bk}. At the same time,
it does not use these theorems, but rather translates the underlying idea
into the context of groups acting on groups: a group operation replaces the sum,
while the action replaces the product. This is a basic theme in
this paper: our solvable group must be separated into a maximal torus,
which acts, and a unipotent group, which is acted upon.

In order to apply Prop.~\ref{prop:generous} our first job is to reduce the question of proving growth for an abstract solvable subgroup of $\GL_r(K)$ to the question of proving growth in a subgroup of a connected solvable linear algebraic group $G$. This reduction is done in \S\ref{s: reduction}. Thus we can assume that $A\subseteq G(K) \leq \GL_r(K)$ where $G=UT$ with $U$ the unipotent radical of $G$ and $T$ a maximal torus of $G$. Our method will be to apply Prop.~\ref{prop:generous} to the natural conjugation action of $T$ on $U$.

Our task is to show that if $A$ does not grow rapidly, then we have two subgroups $S$ and $U_R$ with the given properties. By choosing $G$ suitably we can take $S$ to be $G(K)$, the group $U_R$ is defined at \eqref{eq:astar}. Our job is to show that the group $U_R$ is contained in $A_k$ for some $k\ll_r 1$. In the case where $U$ is abelian this fact follows quite easily from an application of Prop.~\ref{prop:generous} (see \S\ref{s: first interesting}). We make use of the property that all elements of $G(K)$ act on $U(K)$ like elements of the torus. More precisely, for all $g\in G(K)$, there exists $t\in T(K)$ such that $gug^{-1}=tut^{-1}$ for all $u\in U(K)$.

When $U$ is not abelian this property does not hold and we cannot apply Prop.~\ref{prop:generous} directly. Instead we resort to a ``descent'' argument, which we describe in \S\ref{s: descent} (this is the first point where we use the fact that our finite field has prime order). Roughly speaking we obtain the group $U_R$ level-by-level: writing $U=U^0>U^1>\cdots$ for the lower central series of $U$, we observe first that $G/U^1$ has an abelian unipotent radical and so we can apply Prop.~\ref{prop:generous} naively, \`a la \S\ref{s: first interesting}, to obtain the group $U_R/U^1(K)$. Next we consider the quotient $G/U^2$ and we seek to obtain the group $U_R/U^2(K)$. There are three components to this task: we must first construct a set of elements in $A_k$ which act like elements of the torus on $U(K)/U^2(K)$; we then use these elements with Prop.~\ref{prop:generous} to obtain part of $U_R/U^2(K)$; finally there are some elements of $U_R/U^2(K)$ which cannot be obtained this way, but can be obtained as 
commutators of elements in $U_R/U^1(K)$.

Now we repeat this process for subsequent quotients $G/U^3$, $G/U^4$, and so on. Since the nilpotency rank of $U$ is bounded above by $r$ this process terminates after $r$ steps and the result follows. The details of the inductive argument are given in \S\ref{s: proof} where Thm.~\ref{t: main} is proved. In that section we also give a proof of the stronger statement in which $U_R$ is normal in $\langle A \rangle$.

In order to nail down the details of the argument just described we have made use of machinery from the theory of linear algebraic groups. In particular it turns out the exponential map is a very convenient tool, particularly for keeping track of commutators. The theory that we need is given in \S\ref{s: background solvable}. 

The final two sections are devoted to proving the stronger statements given in \S\ref{s: extensions}. In \S\ref{s: extension}, we prove Thm.~\ref{t: main2}; this section is joint work with L\'aszl\'o Pyber and Endre Szab\'o. In \S\ref{s: extension 2}, we prove Thm.~\ref{t: main3}.

\subsection{Generalizations}

\emph{Growth in groups over general finite fields.}
One natural plan is
to extend Thm. \ref{t: main} to the case where $K$ is any finite field
$\mathbb{F}_q$. Indeed, all results in Sections \ref{s: lemmas} to \ref{s: first interesting} of the current document apply in this more general setting.
The difficulties involved in generalizing the rest of the paper to the case
$K = \mathbb{F}_q$, $q=p^\alpha$, $\alpha>1$, 
seem mostly technical; as usual, it may happen that
a second-generation proof will deal with the case $\alpha>1$ automatically
(as was the case for finite simple groups (\cite{bgt2}, \cite{ps2})).

\emph{Growth of finite sets in infinite groups.}
One possible generalization consists in proving Thm.~\ref{t: main} again, as stated,
with $K = \mathbb{Z}/p\mathbb{Z}$ replaced by an infinite field. For $K$ of
characteristic zero, this is in several ways easier than for $K =
\mathbb{Z}/p\mathbb{Z}$: (a) the present proof largely goes through, with
simplifications due to the fact that the finite subgroup structure is much
simpler; (b) real and complex methods are also applicable --
see, e.g., \cite{chang, BG2}.
Here (b) reflects in part the situation in additive combinatorics, where
results on growth in $\mathbb{R}$ are generally older and more direct than
results
on growth in $\mathbb{Z}/p\mathbb{Z}$: over $\mathbb{R}$, one can
exploit an ordering, a metric and a topology that do not exist
over $\mathbb{Z}/p\mathbb{Z}$.

The case of infinite $K$ with positive characteristic cannot really be
easier than the case of $K$ finite, since it contains it as a subcase:
for $K = (\mathbb{Z}/p\mathbb{Z})(T)$ and an algebraic group $G$,
a subset $A\subset G(K)$ could be contained in $G(\mathbb{Z}/p\mathbb{Z})$.
A possible strategy in that case could be to aim to prove a ``reduction"
result, much like Theorem 1 in the present paper: either $A$ grows or it
is essentially contained in $G(K')$, $K'<K$, $K'$ finite (a case which
would then be dealt by a generalisation of the present paper to all finite
fields).

In general, in the present paper, finiteness is a challenge to be coped
with, rather than any sort of key assumption. Non-finiteness, whether of
local or global fields, generally entails additional structure that can do
away with essential difficulties and make multiple approaches possible.
The point here -- as in \cite{helfgott2} and much work since then -- is to
use and develop new techniques that yield growth results even when such
additional structure  is not available.

{\em Flattening of measures in infinite groups.} The other possible
generalization of Theorem 1 to infinite fields is of a stronger kind, viz.,
the kind of generalization pioneered by Bourgain-Gamburd in \cite{bgsu2}.
This involves proving results on the ``flattening" of measures under
convolution rather than on the growth of sets under the group operation.
Such results on measures are particularly useful in proofs of expansion.

In finite fields, statements on growth for sets and the
statements on convolutions of measures are essentially equivalent, as was
shown in \cite[\S 3]{bgexp} (``$\ell_2$ flattening")
by means of the Balog-Gowers-Szemer\'edi theorem. In infinite fields,
results on convolutions of measures are harder. The point of \cite{bgsu2}
(on the group $\SU_2$) is that \cite{helfgott2} is robust enough that,
even though its main result is on finite subsets of finite groups, its
proof can be modified to give a theorem on convolution of measures on an
infinite group of the same Lie algebra type, provided that the distances
among the new elements being constructed are kept track of throughout the
modified proof.

It is our intuition that the ideas in the present paper will yield fruit
in this stronger sense under a treatment similar to that in \cite{bgsu2},
though we have not attempted to do this ourselves.

\subsection{Relation to the previous literature} 

There has been plenty of recent work on growth in solvable and nilpotent groups. Fisher, Katz and Peng \cite{FKP} relate growth in a nilpotent Lie group to growth in its Lie algebra; standard facts about nilpotent algebraic groups (which we outline in Section \ref{s: background solvable}) immediately imply analogous results in the context of nilpotent algebraic groups. Breuillard and Green \cite{BG1} generalised the work of Freiman-Ruzsa and Chang to the torsion-free nilpotent case. Finally Tointon \cite{tointon} has recently proved a Freiman-Ruzsa-type theorem for arbitrary nilpotent groups which, in particular, yields the result of Breuillard and Green as a corollary. As our result is essentially a reduction to the nilpotent case, it complements rather than overlaps with these three articles; indeed we will combine our main result with that of Tointon to prove Thm.~\ref{t: main3}.

While \cite{BG2} treats solvable groups, it is limited to subgroups of $\GL_n(\mathbb{C})$, where the problem yields fairly easily to a direct application of the sum-product theorem in its classical form. The setting of the work of Sanders \cite{sanders} is fairly general, but its conditions are very strong, being of Gromov type.

T.\ Tao proved \cite{taofrei} a structure statement on slowly growing sets
in solvable groups. The main two issues are the following: first, as Tao
directly incorporates ideas from Freiman's theorem, the growth he proves
is at best logarithmic; second, the structure whose presence he proves
(``coset nilprogressions'', \cite[Def.\ 1.11]{taofrei}), besides being
somewhat complicated, involves a series of subgroups $H_{i,0}$ that cannot be
easily quotiented out. A simpler structure (a ``nilprogression'') 
is also shown to exist \cite[Thm.\ 1.17]{taofrei}
 but only for totally torsion-free groups.\footnote{We caution the reader that there are a number of slightly differing definitions of nilprogressions, and coset nilprogressions, in the literature. In particular the definitions used by Tao in his work on solvable groups (which were the first such definitions to appear in the literature) are slightly different from the definitions we use in Thm.~\ref{t: main3}. For that theorem we use the definitions of \cite{bgt11}; see \S\ref{s: extension 2} for full details.}

Using model theory, Hrushovski proved results on slowly growing sets in
$\GL_n(K)$, $K$ any field (see in particular \cite[Cor.\ 5.10]{hrush}). These
results were - like \cite{sanders} - both impressively general and 
quantitatively very weak. Hrushovski's Cor.\ 5.10 is in some sense
orthogonal to most of the work in this paper: it is a
reduction to the solvable setting, whereas our focus will be to reduce the
solvable setting to the nilpotent case.

(The situation is somewhat similar in the case of \cite{bgt11} (based partly on \cite{hrush}), which appeared after the first version of 
the present paper was made publicly available as a preprint. The results
in \cite{bgt11} are very general, to the extent of proving what its authors
call the ``Helfgott-Lindenstrauss conjecture" in a qualitative sense.
However their results are, again, quantitatively very weak. Roughly speaking, they
show that, when some necessary conditions are met, $|A A A|\geq C |A|$ for
$C$ an arbitrarily large constant and
$A$ sufficiently large; in contrast we prove, under stronger conditions including, in particular, an embedding in $\GL_r(K)$, that
$|A A A|\geq C |A|$ with $C = |A|^\delta$, $\delta>0$ a constant.
Indeed it is this form of result that was conjectured by the second author;
see the remarks immediately following \cite[Thm. 1.1]{helfgott3}.)

It is clear that, given our limited state of knowledge on the constants in
Freiman's theorem even in the group $\mathbb{Z}$, any result that includes
cases of relatively rapid growth ($|A|^{1+\delta}\ll |A_3|\ll
|A|^{1+\delta'}$, $\delta, \delta'>0$) must either 
be a reduction to the nilpotent case (like Theorems \ref{t: main} 
and \ref{t: main2}) or have worse-than-polynomial dependence (like Thm.~\ref{t: main3}). A possibility for improvement that might be within reach
could be to strengthen Thm. \ref{t: main3} to give 
$O_r(\exp((\log C)^{O(1)}))$-control, as does Sanders' result \cite{sanders}
over abelian groups; this would, of course, involve strengthening Tointon's
result to give the same kind of control.

Cases $r=2,3$ of Thm. \ref{t: main} were proven in 
\cite[\S 7]{helfgott3}.

\subsection{Acknowledgments}
Pablo Spiga provided help with group theory results; Martin Kassabov provided significant assistance in understanding solvable algebraic groups. Thanks are also due in this regard to Emmanuel Breuillard, Kevin Buzzard, Simon Goodwin, Alex Gorodnik, Scott Murray, L\'aszl\'o Pyber and an anonymous referee. In addition Simon Goodwin pointed out an error in the statement of Lem. \ref{l: kirillov} in an earlier version.

Part of this work was completed while the first author was visiting the University of Western Australia and the University of Bristol; he would like to thank members of both maths departments for providing excellent working conditions, and for their interest in the work at hand. The second author would like to thank the Ecole Polytechnique F\'ed\'erale de Lausanne for hosting him during part of his work on this project.

Section~\ref{s: extension} of this paper is joint work with L\'aszl\'o Pyber and Endre Szab\'o; it is a pleasure to thank them for the warm way in which they have shared their considerable insight.

\section{Background from additive combinatorics}\label{s: lemmas}

Let us establish some notation from additive combinatorics. Our notation in this area is standard and, in particular, is identical to that of \cite{helfgott3}. In this section $G$ is an arbitrary group.

Given a positive integer $k$ and a subset $A$ of a group $G$, we define
$$A_k=\{g_1\cdot g_2\cdots g_k \mid g_i\in A\cup A^{-1}\cup \{1\}\}.$$

Given real numbers $a,b,x_1,\dots, x_n$, we write 
$$a\ll_{x_1,\dots, x_n} b \textrm{  or  } O_{x_1,\dots, x_n}(b)$$
to mean that the absolute value of $a$ is at most the real number $b$ multiplied by a constant $c$ depending only on $x_1,\dots, x_n$. When we omit $x_1, \dots, x_n$, and write $a\ll b$ (or $a=O(b)$), we mean that the constant $c$ is absolute.

\subsection{Growth in subgroups and quotients}

The following basic lemmas relate growth in a group $G$ to growth in
subgroups of $G$, and in quotients of $G$. Citations to \cite{helfgott3} are
given in part for the sake of ease of reference; no doubt many of these results may have been known to specialists for a long time.

We introduce some abuse of notation: For $S,T$ two sets, we write $S\backslash T$ where we mean $S\backslash (S\cap T)$. Similarly if $G$ is a group with $W\subset G, N\lhd G$, then we write $W/N$ where we mean $WN/N$.

The following lemma was first stated and proven in the abelian
case by Ruzsa and Turj{\'a}nyi \cite{ruztur}. The proof carries over to the
nonabelian case; the lemma was stated and proven in full generality in
\cite{helfgott2} and \cite{taononcomm}.

\begin{lem}\label{l: tripling}\cite[Lem. 2.2]{helfgott3}
{\rm (Tripling Lemma).} Let $k>2$ be an integer; let $A$ be a finite subset of a group $G$. 
\begin{enumerate}
\item $\frac {|A_3|}{|A|} \leq \left(3\frac{|A\cdot A \cdot A|}{|A|}\right)^3$
\item $\frac {|A_k|}{|A|} \leq \left(\frac{|A_3|}{|A|}\right)^{k-2}$
\end{enumerate}
\end{lem}

\begin{lem}\label{l: olson}\cite{olson}
Let $A$ be a generating set of a finite group $G$, $B$ a subset of $G$. Suppose that $A$ contains $1$ and $B$ is non-empty. Then $|AB|\geq \min(|B|+\frac12|A|, |G|)$. In particular, if $A\cdot A\cdot A\neq G$ then $|A\cdot A\cdot A|\geq 2|A|$.
\end{lem}

\begin{lem}\label{lem:odt}
Let $H\leq G$ and let $A,B \subset G$ be non-empty finite sets.
Let $l$ be the number of left cosets of $H$ intersecting $A$.
Then
\[|A\cdot B| \geq l |B\cap H|.\]
\end{lem}
\begin{proof}
Let $x_1, x_2,\dotsc, x_l\in A$ be representatives of distinct left
cosets of $H$.
Then
$$|A\cdot B| \geq 
|A\cdot (B\cap H)|\geq \left|\bigcup_{1\leq j\leq l} x_j \cdot (B\cap H)\right|
= l \cdot |B\cap H|.$$
\end{proof}

\begin{lem}\label{lem:duffy} \cite[Lem. 7.2]{helfgott3} 
Let $G$ be a group and $H$ a subgroup thereof. Let $A\subset G$ be a 
non-empty finite set. 
Then
\[|A^{-1} A \cap H| \geq \frac{|A|}{l},\] 
where $l$ is the number of left cosets of $H$ intersecting $A$. 
\end{lem}
\begin{proof} 
By the pigeonhole principle, there is at least one coset $g H$ 
of $H$ containing at least $|A|/l$ elements of $A$ (and thus, in particular, 
at least one element of $A$). Choose an element $a_0 \in g H \cap A$. 
Then, for every $a\in g H \cap |A|$, the element $a_0^{-1} a$ lies both in $H$ and 
in $A^{-1} A$. As $a_0$ is fixed and $a$ varies, the elements $a_0^{-1} a$ 
are distinct.
\end{proof}

The following is a slight generalization of \cite[Lem. 7.3]{helfgott3}.
\begin{lem}\label{l: b1}
Let $H\leq G$ and let $A\subset G$ be a non-empty finite set. Then, for any $k\geq 2$,
$$|A_{k+1}|\geq \frac{|A_k\cap H|}{|A^{-1}A\cap H|}|A|.$$
\end{lem}
\begin{proof}
Let $l$ be the number of left cosets of $H$ intersecting $A$. 
By Lem. \ref{lem:odt} with $B = A_k$,
\[|A_{k+1}| = |A\cdot A_k| \geq l \cdot |A_k \cap H|.\]
Now, by Lem. \ref{lem:duffy},
$|A^{-1} A\cap H|\geq \frac{|A|}{l}$. Hence
$$|A_{k+1}|\geq |A\cdot (A_k\cap H)|\geq l\cdot |A_k\cap H|\geq \frac{|A_k\cap H|}{|A^{-1}A\cap H|}|A|.$$
\end{proof}

We note some other basic results that will be of use later.

\begin{lem}\label{l: b2}\cite[Lem. 7.4]{helfgott3}
Let $H\unlhd G$ and let $\pi:G\to G/H$ be the quotient map. Then, for any finite non-empty subsets $A_1, A_2\subset G$,
$$|(A_1\cup A_2)_4|\geq \frac{|\pi(A_1A_2)|}{|\pi(A_1)|}|A_1|.$$
\end{lem}

\begin{lem}\label{l: b3}
Let $N\unlhd G, R$ a subset of $G$ satisfying $R=R^{-1}$, and $A$ a non-empty finite subset of $G$. Then, for any $C>0$,
$$|AN/N\cap RN/N|\geq \frac{1}{C} |AN/N| 
\implies |A_3\cap RN|\geq \frac{1}{C} |A|.$$
\end{lem}
\begin{proof}
Define $E=A^{-1}A\cap N$.
 Let $g$ be some element of $G$.
Given a fixed element $a_0 \in A \cap g N$, every distinct element $a\in A\cap
g N$ determines a distinct element $a^{-1} a_0$ of $E= A^{-1} A \cap N$. Therefore
$$|E|\geq |A\cap gN|.$$
Thus, for any set $S$ of representatives of the cosets $g N$ with $A\cap g N$
non-empty,
\[|A| = \sum_{g\in S} |A\cap g N| \leq |S| |E| = |AN/N|\cdot |E|.\]
Hence
$$|A_3\cap RN|\geq
 |AN/N\cap RN/N|\cdot |E| \geq \frac{1}{C} |AN/N|\cdot |E| \geq \frac{1}{C}
|A|.$$
\end{proof}

The following lemma is in the spirit of the Cauchy-Davenport theorem \cite[Thm. 5.4]{taovu}.

\begin{lem}\cite[Lem. 2.1]{helfgott3}\label{l: b4}
 Let $A\subseteq G$ with $|A|>\frac12|G|$. Then $A\cdot A=G$.
\end{lem}

\begin{lem}\cite[Lem. 7.6]{helfgott3}\label{l: b5}
 Let $R\subseteq G$ be a subset with $R=R^{-1}$. Let $A\subset G$ be finite; then there is a subset $Y\subset A$ with
$$|Y|\geq\frac{|A|}{|A^{-1}A\cap R|}$$
such that no element of $Y^{-1}Y$ (other than possibly the identity) lies in $R$.
\end{lem}

The next result is a version of Schreier's lemma \cite[\S 4.2]{seress}.

\begin{lem}\label{lem:squid}
Let G be a group. Let $A\subset G$, $H<G$. Suppose $AH/H = G/H$. Then
$\langle A\rangle
= A \cdot \langle A_3 \cap H\rangle$.
\end{lem}
\begin{proof}
Since $AH/H=G/H$, there is an element $a\in A$ lying in $H$, and thus
$e = a \cdot a^{-1}$ is an
element of $A\cdot \langle A^{-1}\cap H\rangle \subset A\cdot \langle
A_3\cap H\rangle$.
It remains to show that, if $a_1\in A\cup A^{-1}$ and $g = a_2 h$,
where $a_2\in A\cup \{1\}$ and
$h\in \langle A_3 \cap H\rangle$, then $a_1 g = a_1 a_2 h$ lies in
$A \cdot \langle A_3 \cap H\rangle$.

Because $AH/H = G/H$, there is an $a_3\in A$
such that $a_1 a_2 H = a_3 H$. Hence $a_3^{-1} a_1 a_2 \in H$, and so
$a_3^{-1} a_1 a_2 \in A_3 \cap H$. Therefore $a_1 a_2 h = a_3 a_3^{-1}
a_1 a_2 h$ lies in
$A\cdot \langle A_3 \cap H\rangle$.
\end{proof}

\subsection{Pivoting}
 The following result is connected to the idea behind a 
{\it sum-product} theorem; 
it relies on the usage in groups of the technique of {\em pivoting}, 
which can in some sense already be found in some proofs of sum-product 
(for instance \cite{gk}) and was developed further in \cite[\S 3]{helfgott3}.
(The same underlying idea was later used in \cite[Lem. 5.3]{bgt2}.)
Note that we never use a sum-product theorem as such.

This proposition is a strengthening of \cite[Cor 3.2]{helfgott3}. 

\begin{prop}\label{prop:generous}
Let $G$ be a group and $\Gamma$ an abelian group of automorphisms of
$G$. Let $X\subset \Gamma$, and set
\[x = |\{y\in X^{-1} X : \text{$y$ has a fixed point other than
$e\in G$}\}|.\]

Then, for any $W\subset G$, either
\begin{equation}\label{eq:hort1}|(X_2(W))_6| \geq \frac{|X|}{x} |W|\end{equation}
or
\begin{equation}\label{eq:hort2}
(X(W))_8 = \langle \langle X\rangle (\langle W\rangle)\rangle.\end{equation}
\end{prop}
Given $A\subset \Gamma$, $B\subset G$, we write $A(B)$ for $\{a(b) :
a\in A, b\in B\}$.
Thus,
$\langle \langle X\rangle (\langle W\rangle)\rangle$ is the 
group generated by all
 elements of the form $y(w)$ with $w\in \langle W\rangle$ and
$y\in \langle X\rangle$.
\begin{proof}
For $\xi \in G$, we define the map
$\phi_\xi : G \times \Gamma \to G$ by
\[\phi_{\xi}(g,\gamma) = g \gamma(\xi).\]

We call $\xi\in G$ a {\em pivot} if, for $g_1,g_2\in W$, 
$\gamma_1,\gamma_2\in X$, we
can have $\phi_{\xi}(g_1,\gamma_1)=\phi_{\xi}(g_2,\gamma_2)$ 
only if $\gamma_1^{-1} \gamma_2$
acts on $G$ with at least one fixed point other than the identity $e\in G$.

By Lem. \ref{l: b5}, there exists a subset $Y\subset X$
with $|Y| \geq |X|/x$ such that no element of $Y^{-1} Y$
(other than possibly the identity) has a fixed point
in $G$ other than the identity.
It is clear that, if $\xi$ is a pivot, then $|\phi_{\xi}(A,Y)| = |Y| |A|$
for any $A\subset G$, and, in particular, for $A=W$.

{\em Case 0: There is a pivot $\xi\in W$.}
Then $\phi_{\xi}(W,Y) \subset (Y(W))_2$, and, at the same time,
$|\phi_{\xi}(W,Y)| = |Y| |W|$. Hence $|(Y(W))_2|\geq |Y| |W|$.

{\em Case 1a: There is a $\xi\in G$, $\xi$ not a pivot, and an $a\in W$
such that $a\xi$ is a pivot.}
Then $|\phi_{a \xi}(W,Y)| = |Y| |W|$. It remains to construct
a subset in $(Y_2(W))_6$ of cardinality $\leq |Y| |W|$. (We can't
assume $\phi_{a \xi}(W,Y)\subset (Y_2(W))_6$ because $\xi$ may not be in $W$.)

Since $\xi$ is not a pivot, there are $g_1, g_2\in W$, $\gamma_1,
\gamma_2\in X$ such that 
$\phi_{\xi}(g_1,\gamma_1)=\phi_{\xi}(g_2,\gamma_2)$
(and so $\gamma_1(\xi) (\gamma_2(\xi))^{-1} = g_1^{-1} g_2$)
 and
$\gamma_1^{-1} \gamma_2$ has $e\in G$ as its only
fixed point in $G$. 

Now, if $x, x'\in G$ satisfy $\gamma_1(x) (\gamma_2(x))^{-1} =
\gamma_1(x') (\gamma_2(x))^{-1}$, then $(x')^{-1} x$ is a fixed
point of $\gamma_2^{-1} \gamma$. Hence $(x')^{-1} x = e$, i.e.,
the map $x\to \gamma_1(x) (\gamma_2(x))^{-1}$ from $G$ to $G$ is injective.

Hence
\[|\{\gamma_1(x) (\gamma_2(x))^{-1} : x\in \phi_{a \xi}(W,Y)\}| = |Y| |W|.\]
Now, for any $g\in W$, $\gamma\in Y$,
\begin{equation}\label{eq:tosco}\begin{aligned}
\gamma_1(\phi_{a\xi}(g,\gamma)) (\gamma_2(\phi_{a\xi}(g,\gamma)))^{-1} 
&= \gamma_1(g \gamma(a\xi)) (\gamma_2(g \gamma(a\xi)))^{-1}\\
&= \gamma_1(g) \gamma(\gamma_1(a\xi) (\gamma_2(a\xi)^{-1}) (\gamma_2(g))^{-1}\\
&= \gamma_1(g) \gamma(\gamma_1(a))
\gamma(\gamma_1(\xi) \gamma_2(\xi)^{-1}) \gamma((\gamma_2(a))^{-1}) 
(\gamma_2(g))^{-1}\\
&= \gamma_1(g) \gamma(\gamma_1(a))
\gamma(g_1^{-1} g_2) \gamma((\gamma_2(a))^{-1}) (\gamma_2(g))^{-1}\\
&\in (Y_2(W))_6 \subset (X_2(W))_6,\end{aligned}\end{equation}
where we have used the fact that $\Gamma$ is abelian. (What we have
done in (\ref{eq:tosco})
is apply the map $x\to \gamma_1(x) (\gamma_2(x))^{-1}$ to
$\phi_{a\xi}(g,\gamma)$ so as to get rid of $\xi$.)

Therefore, $|(X_2(W))_6|\geq |Y| |W|$.

{\em Case 1b: There is a $\xi\in G$, $\xi$ not a pivot, and a $y\in X$
such that $y(\xi)$ is a pivot.}
Then $|\phi_{y(\xi)}(W,Y)| = |Y| |W|$.
Much as in the previous case, we have
Hence
\[|\{\gamma_1(x) (\gamma_2(x))^{-1} : x\in \phi_{y(\xi)}(W,Y)\}| = |Y| |W|.\]
Now, for any $g\in W$, $\gamma\in Y$,
\begin{equation}\label{eq:tosca}\begin{aligned}
\gamma_1(\phi_{y(\xi)}(g,\gamma)) (\gamma_2(\phi_{y(\xi)}(g,\gamma)))^{-1} 
&= \gamma_1(g) \gamma(\gamma_1(y(\xi)) (\gamma_2(y(\xi))^{-1}) (\gamma_2(g))^{-1}\\
&= \gamma_1(g) \gamma(y(\gamma_1(\xi) \gamma_2(\xi)^{-1}))
(\gamma_2(g))^{-1}\\
&= \gamma_1(g) 
\gamma(y(g_1^{-1} g_2)) (\gamma_2(g))^{-1}
\in (Y_2(W))_4 \subset (X_2(W))_4.\end{aligned}\end{equation}

Therefore, $|(X_2(W))_4|\geq |Y| |W|$.

{\em Case 2: No element $\xi \in \langle
\langle X\rangle(\langle W\rangle)\rangle$ is a pivot.}
This means that for every $\xi\in
\langle X\rangle(\langle W\rangle)\rangle$ there are 
$g_1, g_2\in W$, $\gamma_1, \gamma_2 \in X$ such that 
$\gamma_1(\xi) (\gamma_2(\xi))^{-1} = g_1^{-1} g_2$
 and
$\gamma^{-1} \gamma_2$ has $e\in G$ as its only
fixed point in $G$. 

As said before, the map $x \to \gamma_1(x) (\gamma_2(x))^{-1}$ is
injective provided
$\gamma^{-1} \gamma_2$ has $e\in G$ as its only
fixed point in $G$. 
Hence, given $g_1,g_2\in W$, $\gamma_1,\gamma_2\in Y$,
$\gamma_1\ne \gamma_2$,
there is at most one $\xi \in \langle
\langle X\rangle(\langle W\rangle)\rangle$ such that
$\gamma_1(\xi) (\gamma_2(\xi))^{-1} = g_1^{-1} g_2$. This, together
with the fact that there are such $g_1,g_2$, $\gamma_1,\gamma_2$ for
every $\xi \in \langle
\langle X\rangle(\langle W\rangle)\rangle$, already implies that
\begin{equation}\label{eq:tristol}
|Y| |W| \geq 
|\langle \langle X\rangle(\langle W\rangle)\rangle|,\end{equation}
i.e., $Y$ and $W$ are large.

We can prove more. Let 
\[R_{\xi} = \{(g_1,g_2,\gamma_1,\gamma_2)\in W\times W\times
Y\times Y : \gamma_1\ne \gamma_2, g_1 \gamma_1(\xi)
= g_2 \gamma_2(\xi) \}
\]
We have already shown that the sets $R_{\xi}$ are disjoint
as $\xi$ ranges in $G$. Choose $\xi_0\in 
\langle X\rangle(\langle W\rangle)\rangle$ such that
$|R_{\xi_0}|$ is minimal. Then
\[|R_{\xi_0}| \leq \frac{|W|^2 |Y| (|Y|-1)}{
|\langle \langle X\rangle(\langle W\rangle)\rangle|}
< \frac{|W|^2 |Y|^2}{
|\langle \langle X\rangle(\langle W\rangle)\rangle|}\]
and so
\[|\{(g_1,g_2,\gamma_1,\gamma_2)\in W\times W\times
Y\times Y : g_1 \gamma_1(\xi_0)
= g_2 \gamma_2(\xi_0) \}| 
< \frac{|W|^2 |Y|^2}{
|\langle \langle X\rangle(\langle W\rangle)\rangle|} + |W| |Y|.\]

By Cauchy-Schwarz,
\[\begin{aligned}
(|W| |Y|)^2 &= \left(\sum_{r\in W Y(\xi_0)} |\{(g,\gamma)\in
W\times Y: g \gamma(\xi_0) = r\}|\right)^2\\ &\leq
|W Y(\xi_0)| \cdot \sum_{r\in W Y(\xi_0)} |\{(g,\gamma)\in
W\times Y: g \gamma(\xi_0) = r\}|^2\\
&= |W Y(\xi_0)| 
|\{(g_1,g_2,\gamma_1,\gamma_2)\in W\times W\times
Y\times Y : g_1 \gamma_1(\xi_0)
= g_2 \gamma_2(\xi_0) \}|,\end{aligned}\]
and so 
\[|W Y(\xi_0)| > \frac{|W|^2 |Y|^2}{
\frac{|W|^2 |Y|^2}{
|\langle \langle X\rangle(\langle W\rangle)\rangle|} + |W| |Y|}
\geq \frac{|W|^2 |Y|^2}{
2 \frac{|W|^2 |Y|^2}{
|\langle \langle X\rangle(\langle W\rangle)\rangle|}} =
\frac{1}{2}
|\langle \langle X\rangle(\langle W\rangle)\rangle|,\]
where we are using (\ref{eq:tristol}).

Now, recall that $\xi_0$ is not a pivot.
Hence there are $g_1, g_2\in W$, $\gamma_1,
\gamma_2\in X$ such that 
$\gamma_1(\xi_0) (\gamma_2(\xi_0))^{-1} = g_1^{-1} g_2$
 and
$\gamma_1^{-1} \gamma_2$ has $e\in G$ as its only
fixed point in $G$. Proceeding as before, we have
\[|\{\gamma_1(x) (\gamma_2(x))^{-1} : x\in \phi_{\xi_0}(W,Y)\}| 
= |\phi_{\xi_0}(W,Y)| >
\frac{1}{2}
|\langle\langle X\rangle(\langle W\rangle)\rangle|.\]

Now, much as before,
we see that, for any $g\in W$, $\gamma\in Y$,
\begin{equation}\begin{aligned}
\gamma_1(\phi_{\xi}(g,\gamma)) (\gamma_2(\phi_{\xi}(g,\gamma)))^{-1} 
&= \gamma_1(g) \gamma(\gamma_1(y(\xi)) (\gamma_2(y(\xi))^{-1}) (\gamma_2(g))^{-1}\\
&= \gamma_1(g) \gamma(y(\gamma_1(\xi) \gamma_2(\xi)^{-1}))
(\gamma_2(g))^{-1}\\
&= \gamma_1(g) 
\gamma(g_1^{-1} g_2) (\gamma_2(g))^{-1}
\in (Y(W))_4 \subset (X(W))_4.\end{aligned}\end{equation}

Hence
\[|(X(W))_4|> \frac{1}{2} |\langle\langle X\rangle(\langle W\rangle)\rangle|\]
and so, by Lem. \ref{l: b4},
\[(X(W))_8 = \langle\langle X\rangle(\langle W\rangle)\rangle.\]
\end{proof}

\section{Background on solvable groups}\label{s: background solvable}

Let $K$ be a finite field of characteristic $p$ and $K'$ some finite extension of $K$. If $H$ is an algebraic group defined over $K'$, then we call $H$ a {\it $K'$-group}. Now let $G$ be a connected solvable algebraic $K$'-subgroup of $\GL_r$. We are interested in studying $G(K')\cap \GL_r(K)$.

Recall that a Borel subgroup of $\GL_r$ is a closed, connected, solvable
 subgroup $B$ of $\GL_r$, which is maximal for these properties. So, in particular, $G$ is contained in a  Borel subgroup of $\GL_r$. Let $B$ and $B_1$ be two Borel subgroups of $G$; a classic result of algebraic groups says that,  $B(\Kbar)$ and $B_1(\Kbar)$ are conjugate in $\GL_r(\Kbar)$, and in particular are conjugate into the set of upper triangular matrices (see for instance \cite[6.2.7]{springer}).

We say that $G$ is called $K'$-split if it has a composition series $G=G_0\supset G_1 \supset \cdots \supset G_s = \{1\}$ consisting of connected $K'$-subgroups such that $G_i/ G_{i+1}$ is $K'$-isomorphic to $G_a$ or $\GL_1$ \cite[15.1]{borel}. 

We say that $G$ is {\em trigonalizable} over $K'$ if
there exists $x\in \GL_r(K')$  such that $xGx^{-1}$ consists of 
upper-triangular matrices. 
Since $K'$ is finite, $G$ is trigonalizable over $K'$
if and only if $G$ is $K'$-split. What is more, 
every image of $G$ under a $K'$-morphism is $K'$-split \cite[15.4]{borel}.

We can write $G = UT$, where $U$ is unipotent (it is the {\it unipotent radical} of $G$), $T$ is a torus, and both are defined over $K'$
\cite[10.6]{borel}. The groups $U$ and $T$ are $K'$-split if and only if $G$ is $K'$-split. Furthermore, if $U$ is $K'$-split, then any subgroup of $U$ that is defined over $K'$ is $K'$-split. Note too that $U$ is connected \cite[6.3.3]{springer}.

We introduce two assumptions for this section: firstly we assume that $G$ is trigonalizable (and hence $K'$-split) over $K'$ (recall that $K'$ is a finite extension of $K$). Secondly we assume that $p>r$; this implies that $U(\Kbar)$ is a group {\it of exponent $p$}; that is to say, $u^p = 1$ for all $u\in U(\Kbar)$.

Before we proceed we note an abuse of notation: for a variety $V$ defined over $K$, and a subvariety $W/\overline{K}$ defined over the algebraic completion $\overline{K}$ of $K$, we will write $W(K)$ for $W(\overline{K})\cap V(K)$. (We will even speak of the points of $W$ over $K$, meaning $W(K):= W(\overline{K})\cap V(K)$.)

\subsection{Central series, and a more general definition of $G$}

For subgroups $A$ and $B$ of an abstract group $H$ we define
$$[A,B] = \langle [a,b] \, \mid \, a\in A, b\in B \rangle.$$
Define the {\it lower central series of} $H$ to be the series 
$$H=H^0\geq H^1\geq H^2 \geq \cdots,$$
where $H^{i+1} = [H, H^i]$ for $i=0,\dots$.

In this way we can define a lower central series for $U(\Kbar)$; each member of the resulting series of abstract groups turns out to be the set of points over $\Kbar$ of a family of $K'$-groups, $U^0, U^1, \dots$ \cite[2.3]{borel}. We therefore define $U=U^0\geq U^1 \geq\cdots$ to be the lower central series of $U$.

Let $s$ be the nilpotency rank of $U$; i.e.,  $s$ is the smallest number such that $U^s=\{1\}$. Since $G(K)$ lies inside $B(\Kbar)$, and $B(\Kbar)$ has nilpotency rank $r-1$, we conclude that $G(K)$ has nilpotency rank at most $r-1$. Note that $T$ normalizes $U^i$ for all $i$, and $U^i$ is $K'$-split for every $i$.

By definition the quotient $U^i/ U^{i+1}$ is an abelian group that is $K'$-split. It is, therefore, isomorphic to $\underbrace{G_a\times \cdots \times G_a}_{t}$ \cite[14.3.7]{springer}. If $G=B$ it is obvious that $t = r-i-1$ for $i=0,\dots, r-2$. Since $G<B$, we conclude that $t<\frac12r^2$ for all $i$.

It will be useful to prove results when $G$ is not just a subgroup of $\GL_r$, but a quotient of subgroups. Specifically, let $H$ be a connected solvable subgroup of $\GL_r$ defined over a finite extension $K'$ of $K$. Write $H=UT$, as above; define $G=H/U^i$, where $U^i$ is a group in the lower central series of $U$. Then $G$ is connected and solvable, and defined over $K'$. 

The statements that we have made so far in this section all apply in this more general setting. We work in this more general setting for the remainder of the section.

\subsection{The Lie algebra and $\exp$}

We can associate to our linear algebraic group $G$ (resp. $U$, $T$) a Lie
algebra $\mathfrak{g}$ (resp. $\mathfrak{u}$, $\mathfrak{t}$)
in the usual way. We will make frequent use of the adjoint representation ${\rm Ad}:G\to \GL(\lieg)$.

Write $U_r$ for the unipotent radical of $B$, the Borel subgroup containing $G$; let $\lieu_r$ be the Lie algebra of $U_r$. We are able to define the exponential and logarithm map
\begin{equation}\label{e: exp}
  \exp: \lieu_r\to U_r, \, \, X \mapsto \sum_{i=0}^\infty \frac{X^i}{i!}, \textrm{ and  }
\log: U_r\to \lieu_r, \, \, x \mapsto \sum_{i=1}^\infty (-1)^{i+1} \frac{(x-1)^i}i
\end{equation}
in the usual way. Observe that all elements $X$ in $\lieu$ satisfy $X^r=0$ and all elements $x$ in $U$ satisfy $(x-1)^r=0$. Thus these maps are polynomials defined over $\Z[\frac1{r!}]$; in particular, since $p>r$, $\exp$ and $\log$ are defined over $\Z/p\Z$.

The Lie algebra $\lieu_r$ is, by definition, a vector space over $\Kbar$; the Lie algebra, $\lieu$, is a subalgebra of $\lieu_r$ and is also a vector space; in particular, $\lieu$ is an affine algebraic variety defined by a finite set of linear equations. We can, therefore, write $\lieu(L)$ for the set of points of $\lieu$ over some field $L$. Note, then, that $\lieu$ and $\lieu(\Kbar)$ coincide.

We list some standard properties of the exponential map; since we are working with matrix groups, these may be verified directly using (\ref{e: exp}). (In the context of Lie groups, these properties can be used to define the exponential map c.f. \cite[Thm. 3.7]{kirillov}.)

\begin{lem}\label{l: kirillov}
Take $X\in \lieu_r(\Kbar)$, $c_1, c_2\in\Kbar$. Then
\begin{enumerate}
\item\label{l1} $\exp((c_1+c_2)X) = (\exp(c_1X))(\exp(c_2X))$;
\item\label{l2} $\exp(-X) = (\exp X)^{-1}$;
\item\label{l4} The map $\exp:\lieu_r(\Kbar)\to U_r(\Kbar)$ is a bijection, with inverse equal to $\log$.
\end{enumerate}
\end{lem}

For fixed $X\in\lieu_r(\Kbar)$, define the map
\begin{equation}\label{e: 1 parameter}
\phi_X(t):K\to U_r, \, t\mapsto \exp(tX).
\end{equation}
Item (\ref{l1}) implies that this map is a morphism of linear algebraic groups (a so-called {\it 1-parameter subgroup}); the image of $\phi_X$ is a 1-dimensional subgroup $R$ of $U$ and, differentiating with respect to $t$ one sees that, $d\phi_X(0)=X$. Simple matrix calculations yield that this property uniquely defines the 1-parameter subgroup. (Note that, from here on, we will refer to both $\phi_X$, and the image of $\phi_X$, as a 1-parameter subgroup.)

Choosing $X$ in $\lieu(K')$, for some field $K'$, and using the fact that $\exp$ is defined over $\Z/p\Z$, we conclude that $R$ is a $K'$-group. We will use \cite[2.2]{borel} to generalize this observation to groups generated by (the images of) 1-parameter subgroups.

The restriction of $\exp$ to the Lie algebra $\lieu$ is not, in general, a map into $U$.\footnote{Our thanks to Simon Goodwin for pointing this out.} However, for a sufficiently ``nice" embedding of $G$ in $\GL_r$ this property can hold; we follow McNinch \cite{mcninch} in referring to this as an {\it exponential type representation}. Note that, in this case, the map $exp: \lieu\to U$ is injective since it is a restriction of the injective map $\exp:\lieu_r\to U_r$.

\begin{lem}\label{l: kirillov2}
Suppose that $G$ is of exponential type in $\GL_r$ and let $\phi:U\to U$ be a morphism of algebraic groups defined over a field $K'$; write $d\phi:\lieu\to\lieu$ for the derivative at the identity. Then
$$\phi(\exp X) = \exp(d\phi(X)).$$
In particular, for $t\in T(\Kbar)$,
$$t(\exp X)t^{-1} = \exp ({\rm Ad}(t)(X)).$$
\end{lem}
\begin{proof}
Fix $X\in\lieu$ and take $s\in\Kbar$. Observe that $\phi(\exp(sX))$ is a 1-parameter subgroup in $U$ with tangent vector at identity $d\phi (d\exp (X)) = d\phi (X)$. Thus, by the uniqueness of 1-parameter subgroups $d\phi(\exp(sX))$ = $\exp(sd\phi(X))$.
\end{proof}

We will assume from here on that $G$ is of exponential type in $\GL_r$. Now recall that the unipotent radical $U$ of $G$ is defined over a finite field $K'$. 

\begin{lem}\label{l: bijection}
Let $L$ be a finite field contained in, equal to, or containing $K'$. The map
$\exp:\lieu(L)\to U(L)$ is a bijection.
\end{lem}
\begin{proof}
The map $\exp:\lieu\to U$ is defined over $\Z$, and so maps elements of $\lieu(L)$ to elements of $U(L)$; in other words the map $\exp:\lieu(L)\to U(L)$ is well-defined.

By definition the algebra $\lieu$ lies inside $\lieu_r$ a maximal unipotent lie subalgebra of $\mathfrak{gl}_r$. The map $\exp$, as we have defined it, is a restriction of the map $\exp: \lieu_r\to U_r$, where $U_r$ is a maximal unipotent subgroup of $\GL_r$. 

The map $\exp: \lieu_r(L)\to U_r(L)$ is an injection, hence the same can be said for the restriction $\exp: \mathfrak{u}(L)\to U(L)$. If $L$ contains $K'$, then $|\lieu(L)|=|U(L)|$, and so the map $\exp$ is a surjection as required.

We must prove that $\exp$ is a surjection when $L$ is contained in $K'$. It is sufficient to prove that if $X\in \lieu(K')\backslash\lieu(L)$, then $\exp(X)\not\in U(L)$. If we represent $X$ as a strictly upper-diagonal matrix with some entries not contained in $L$, then this follows directly from the definition of $\exp$, equation (\ref{e: exp}).
\end{proof}

\subsection{Weights and roots}

If $H$ is a closed subgroup of $U$ that is normalized by $T$, then $\lieh$, the Lie algebra of $H$, is also $T$-invariant (under the adjoint representation). This allows us to define {\bf weights} and {\bf roots} for the group $G$. We proceed in a similar way to \cite[8.17]{borel}. 

The group $T$ acts on $\lieu$ (considered as a vector space over $\Kbar$) so we have a rational representation of $T$; then we can decompose $\mathfrak{u}$ into weight spaces:
$\lieu_\alpha = \{v\in \lieu \, \mid \, \Ad(t)v=\alpha(t)v \textrm{ for all } t\in T\}$.
Here $\alpha: T\to \GL_1$ is a character of $T$. Those $\alpha$ for which $\lieu_\alpha\neq\{0\}$ are called the {\bf weights} of $T$ in $\lieu$. We write $\Phi$ for the set of weights of $T$ in $\lieu$; then 
$$\lieu = \oplus_{\alpha\in\Phi} \lieu_\alpha.$$
We allow the possibility that $\alpha$ is the trivial weight. We will write $\Phi^*$ for the set of non-trivial weights in $\Phi$, we call $\Phi^*$ the set of {\bf roots} of $G$ relative to $T$.

Note first that if $\alpha$ is defined over a field $K'$, then $\lieu_\alpha$ is defined over $K'$. On the other hand observe that $\lieu_\alpha$ is not necessarily a subalgebra of $\lieu$ (since it may not be closed under $[\,,\,]$). However any $1$-dimensional subspace of $\lieu$ is a subalgebra of $\lieu$ (since $[u,ku] = 0$ for every $u\in \lieu, k\in\Kbar$).

\subsection{Weight and root subgroups}

We reiterate that the group $G$ is of exponential type in $\GL_r$. A {\it weight subgroup} of $U$ is a 1-parameter subgroup $R$ that is defined over $K'$, and is normalized by $T$. Since $R$ is normalized by $T$, the Lie algebra 
$\mathfrak{r}$ of $R$ is also $T$-invariant. In other words $\mathfrak{r}$ lies inside $\lieu_\alpha$ for some weight $\alpha$ of $T$ in $\lieu$. We write $\alpha(R)$ for the weight associated with a weight subgroup $R$.

If $\alpha(R)\in\Phi^*$ (i.e.,  $\alpha(R)$ is a root), then we call $R$ a {\bf root subgroup}.

\begin{lem}\label{l: roots}
Let $U=V^0> V^1 > V^2 > \cdots V^s = \{1\}$ be a series of closed connected normal $K'$-subgroups of $G$ of exponential type in $\GL_r$ such that $V^i/ V^{i+1}$ is abelian for $i=0,\dots, s-1$. There exist a finite set of weight subgroups $R_1,\dots, R_d$ in $U$ where $d=\dim U$, such that any element $u\in U(\Kbar)$ can be written $u=r_1 \cdots r_d$ and $r_i\in R_i(\Kbar)$ for $i=1,\dots, d$. 

The weight subgroups can be chosen so that
\begin{enumerate}
\item\label{n-1} $R_l \cdots R_d$ is a normal subgroup of $G$ for $l=1,\dots, d$;
\item\label{no} $R_j(\Kbar)\cap (R_l(\Kbar) \cdots R_d(\Kbar)) = \{1\} $ for $j<l$;
\item\label{n+1} the representation of $u$ is unique.
\item\label{n+2} there exist integers $1=d_0 < d_1 < \cdots < d_{s-1}\leq d$ so that
$$U^i = R_{d_i}R_{d_i+1} \cdots R_d,$$
for $i=0,\dots, s$.
\end{enumerate}
\end{lem}
\begin{proof}
If $U$ has dimension $1$, then define $R_1=U$, and we are done. Now proceed by induction on the dimension of $U$. Then we can assume that root groups exist for $V^1$ satisfying the four given properties; label these weight groups $R_{e+1}, \dots, R_d$. In addition write $\liev^1$ for the Lie algebra of $V^1$.

For each $\alpha\in\Phi$ we can write $\liev_\alpha = \liex_\alpha\oplus \liew_\alpha$ where $\liex_\alpha = \liev_\alpha\cap \liev^1$ and $\liew_\alpha$ is a complement. Define
$$\Phi^1 = \{\alpha\in\Phi \, \mid \, \liev_\alpha\neq
\liex_\alpha\},$$
i.e.,  the set of roots whose root spaces do not lie wholly within $\liev^1$.
Then we can decompose $\liev$ as follows:
$$\liev = \liev^1 \bigoplus \left(\bigoplus\limits_{\alpha\in\Phi^1} \liew_\alpha\right).$$

Now we construct our root groups: we choose a basis for each $\liew_\alpha$, we let $\{v_1, \dots, v_e\}$ be the union of these bases and then set $\liew_i = \langle v_i \rangle$ for $i=1,\dots, e$. Define $R_i$ to be the 1-parameter subgroup given by $v_i$, i.e.,  $R_i(\Kbar)=\exp(\liew_i(\Kbar))$ is a closed 1-dimensional subgroup with $\liew_i$ as a Lie algebra.

Now observe that the subgroups $R_i$ are normalised by $T$; then, since $U/V^1$ is abelian, we obtain that 
\[R_{l} R_{l+1} \dotsb R_{e-1} V^1 /V^1\]
is a group for any $i\geq 0$, $l\geq 1$ with $d_i\leq l < d_{i+1}$. Now, since $V^1 = R_eR_{e+1} \dotsb R_d$, we obtain that (\ref{n-1}) holds.

Let us now prove property (\ref{no}). If a closed (i.e., algebraic) subgroup $H_1$
of an algebraic group $H$ normalizes a closed subgroup $H_2$ of $H$,
and both $H_1$ and $H_2$ are connected, then $H_1 H_2 = H_2\rtimes H_1$
is a closed, connected subgroup of $H$ \cite[\S 7.5]{humphreys3}. Hence $R_l\dotsc R_d$ is a closed
connected subgroup of $U$.
Now apply the inverse of $\exp$ to $R_j(\Kbar)$ (with $j<l$) and to $R_l\cdots R_d(\Kbar)$ to yield the respective Lie algebras. By construction the intersections of these Lie algebras is $\{0\}$. Since the $\exp$ map is one-to-one, we conclude that $R_j(\Kbar)\cap R_l\cdots R_d(\Kbar) = \{1\}$ as required.

We need to prove uniqueness. We proceed by induction on the dimension of $U$. Clearly the statement is true if this dimension is equal to $1$; now suppose that $\dim U = d$, and suppose that $r_1 \cdots r_{d-1} r_d = r_1' \cdots r_{d-1}' r_d'$; then 
$$(r_1')^{-1}r_1r_2\cdots r_{d-1}=r_2'\cdots r_{d-1}' \in R_2\cdots R_{d}.$$
 Now $r_1 (r_1')^{-1} \in R_2\cdots R_{d} \cap R_1$, thus $r_1 =r_1'$ by property (\ref{no}). This implies that $r_2\cdots r_{d} = r_2'\cdots r_{k}'$, and the result follows by induction.

Finally (\ref{n+2}) follow by construction.
\end{proof}

\begin{cor}\label{c: pajama}
Let $R$ be a weight subgroup of $G$ defined over a field extension $K'$ of $K$. Then
\begin{enumerate}
\item $|R(K)|\leq |K|$;
\item Define a character $\beta: T\to \GL_1$ via
$$tx_R(s)t^{-1} = x_R(\beta(t) s).$$
Then $\beta=\alpha(R)$. 
\end{enumerate}
\end{cor}
\begin{proof}
Let $\mathfrak{r}$ be the Lie algebra of $R$. Lem. \ref{l: bijection} implies that the $\exp$ map induces a one-to-one correspondence between the number of points in $\mathfrak{r}(K)$ and $R(K)$. Now $\mathfrak{r}$ is a 1-dimensional subspace of $\lieu$, hence there is a $v\in\lieu$ such that
$$\mathfrak{r}(K) = \{kv |k\in\Kbar\}\cap \lieu(K).$$
Clearly it is not possible for there to be more than $|K|$ elements in this set.

The second property is a consequence of Lem.\ \ref{l: kirillov2}.
\end{proof}

\subsection{Height and standard form}\label{s: height}

Lem. \ref{l: roots} allows us to make a number of useful definitions. We apply the Lem. \ref{l: roots} to the group $G$; the groups $V^i$ in (\ref{n+2}) are prescribed to be members of the lower central series of $U$; in other words $V^i=U^i$ for $i=0,1,\dots$. We now define $\Phi_R$ to be a set of weight subgroups for $G$ that satisfy Lem. \ref{l: roots} in this setting. 

Note that, to apply Lem. \ref{l: roots} in this way we need to be sure that $U^i$ is of exponential type in $\GL_r$ for each $i$. This is clear enough: for some (finite) $s$, $U^s$ is trivial, and the result holds. Then, the result holds for $U^{s-1}$ by an argument similar to that given in Lem. \ref{l: alex}. A repetition of this argument in the quotient $U/U^{s-1}$ achieves the same result for $U^{s-2}$, and so on.

Write $\Phi^*_R$ for the set of root subgroups in $\Phi_R$. Recall that $d=\dim U$, and note that there may be more than $d$ weight subgroups in $U$; $\Phi_R$ does not necessarily contain all of them.

Lem. \ref{l: roots} yields a natural notion of {\it height} in $\Phi_R$; let the height 
${\rm ht}(R)$ of a weight subgroup $R$ be 
${\rm ht}(R)=i$, where $i$ is the first member of the derived series of $U$ that does not contain $R$. Thus we will have weight subgroups of heights $1,\dots, s$. Observe that, by construction, the weight subgroups $R_1, \dots, R_d$ are ordered by increasing height; in other words, they satisfy ${\rm ht}(R_i)\leq {\rm ht}(R_{i+1})$ for $i=1,\dots, d-1$. If $\Sigma$ is a subset of $\Phi_R$ then write
$$\Sigma^i = \{R\in \Sigma_R \, \mid \, {\rm ht}(R)=i\}.$$

The lemma also allows us to consider a {\it standard form} for an element $g\in G(\Kbar)$. We write
\begin{equation}\label{standard}
g = x_{R_1}(s_1)\cdots x_{R_d}(s_d) t
\end{equation}
where $t\in T(\Kbar)$, $s_i \in \overline{K}$
 and $x_{R_i}(s_i) = \exp(s_i v_i) \in R_i(\Kbar)$ for $i=1,\dots, d$. It will be convenient for us to define a function 
$$t_{R_i}: G\to \Kbar, \, g \mapsto s_i,$$
where $s_i$ is the corresponding element of $\Kbar$ given in the standard form for $g$, as in (\ref{standard}).

\subsection{The groups $U_L, U_\Lambda, U_R$, and $E$}

Consider the lower central series for $G$; once again this is a series of connected normal $K'$-subgroups of $G$: 
$$G=G_0> G_1\geq G_2 \geq \cdots.$$ 
Note that $G_i\leq U$ for all $i\geq 1$. Since $G$ is not
in general nilpotent, we define $U_L$ to be the last term in the lower central series for $G$; that is $U_L=G^i$ where $[G, G^i] = G^i$. Alternatively $U_L$ can be thought of as the smallest normal subgroup of $G$ such that $G/ U_L$ is nilpotent. By definition $U_L$ is $K'$-split and connected.

Define $\Lambda = \Phi_R\backslash\Phi_R^*$; in other words $\Lambda$ is the set of weight subgroups in $\Phi_R$ that are not root subgroups:
$$\Lambda =\{ R_i \, \mid \, \alpha(R_i) = 1\}.$$
Now we define 
$$U_\Lambda = \langle R \, \mid \, R \textrm{ is a weight subgroup, and } \alpha(R)=1\rangle.$$
The next couple of results give information about the group $U_\Lambda$. 

\begin{lem}\label{l: conj lambda}
 Take $u\in U_\Lambda$. Let $R$ be a weight subgroup of $G$. Then $uRu^{-1}$ is a weight subgroup of $G$ and $\alpha(R) = \alpha(uRu^{-1})$.
\end{lem}
\begin{proof}
Consider $tuRu^{-1}t^{-1}$ for $t\in T$:
$$tuRu^{-1}t^{-1} = (tut^{-1}) (tRt^{-1}) (tu^{-1}t) = uRu^{-1}.$$ 
Thus $uRu^{-1}$ is a weight subgroup of $U$. 

Recall that we write $x_R(s)$ for an element of the weight subgroup $R$, with $s$ an element of $\Kbar$. The weight subgroup $uRu^{-1}$ has elements $ux_R(s)u^{-1}$, with the map 
$$\Kbar \to uRu^{-1}, \, s\mapsto ux_R(s)u^{-1}$$
 an isomorphism. Then
$t(ux_R(s)u^{-1})t^{-1} = ux_R(\alpha(R)(t)s)u^{-1}$, and so $\alpha(R)= \alpha(uRu^{-1})$ as required.
\end{proof}

Given a group $G$ and $H_1,H_2<G$, we write $C_{H_1}(H_2)$ for the
intersection
$H_1 \cap C_G(H_2)$ of $H_1$ with the centraliser $C_G(H_2)$ of $H_2$.

\begin{lem}\label{l: desc}
 $C_{U(\Kbar)}(T(\Kbar)) = U_\Lambda(\Kbar) = R_{i_1}(\Kbar) \cdots R_{i_l}(\Kbar)$ where $\Lambda = \{R_{i_1}, \dots, R_{i_l}\}$.
\end{lem}
\begin{proof}
 Lem.\ \ref{l: conj lambda} implies that $R_{i_1}\cdots R_{i_l}$ is a group. Since $R_{i_j}\in\Lambda$ implies that $R_{i_1}(\Kbar)$ clearly centralizes $T(\Kbar)$ we conclude that
$$R_{i_1}(\Kbar) \cdots R_{i_l}(\Kbar) \leq U_\Lambda(\Kbar)\leq C_{U(\Kbar)}(T(\Kbar))$$
Hence it is sufficient to prove that $C_{U(\Kbar)}(T(\Kbar))\leq R_{i_1}(\Kbar)\cdots R_{i_l}(\Kbar)$. 

Now suppose that $u\in C_{U(\Kbar)}(T(\Kbar))$ and $u\not\in R_{i_1}(\Kbar)\cdots R_{i_l}(\Kbar)$. Then, by Lem. \ref{l: roots}, 
$u=x_{R_1}(s_1)\cdots x_{R_k}(s_k)$ for some $s_i\in \Kbar$. By assumption $s_j\neq 0$ for some $j$ such that $\alpha(R_j)\neq 1$. But this implies that 
$tut^{-1}=x_{R_1}(s_1')\cdots x_{R_k}(s_k')$ with $s_j \neq s_j'$. Since the expression for $tut^{-1}$ is unique, we conclude that $tut^{-1}\neq u$ which is a contradiction.
\end{proof}

Now consider $C_{G(\Kbar)}(T(\Kbar))$; it turns out that this group is the set of points over $\Kbar$ for a $K'$-subgroup of $G$ \cite[18.2]{borel}. We denote this $K'$-subgroup of $G$ by $E$; it is a {\it Cartan subgroup} of $G$, and is a maximal connected nilpotent $K'$-subgroup in $G$ \cite[12.1]{borel}.

\begin{cor}
The Cartan subgroup $E$ satisfies $E = T\times U_\Lambda$. 
\end{cor}

Note that, since $E$ is a $K'$-group, we conclude that $U_\Lambda$ is a $K'$-group (both are, therefore, $K'$-split) \cite[15.4, 15.5]{borel}. Furthermore, Lem. \ref{l: desc} implies that $U_\Lambda$ is of exponential type in $\GL_r$; the same can be said, therefore, of $E$.

Now we turn our attention from those weight subgroups that are not root subgroups, to those that are. We define $U_R$ to be the subgroup of $U$ that is generated by root subgroups:
\begin{equation}\label{eq:astar}
U_R = \langle R \, \mid \, R\in\Phi_R^* \rangle.\end{equation}

\begin{lem}
$U_R$ is normal in $G$.
\end{lem}
\begin{proof}
Take $g\in G$ and write $g$ in standard form:
$$g=x_{R_1}(s_1)\cdots x_{R_d}(s_d)t.$$
Let $R_i$ be a root subgroup, and take $r\in R_i$; it is sufficient to prove that $grg^{-1}\in U_R$.

It is easy to see that this reduces to showing that
$x_{R_i}(s_i)x_{R_j}(s_j)x_{R_i}(-s_i)$ is in $U_R$, where $R_i\in\Lambda$, and $R_j\in \Phi_R^*$. This result follows from Lem. \ref{l: conj lambda}.
\end{proof}

We want to connect our understanding of the groups $U_L, U_R$, and $E$; first an easy technical lemma.

\begin{lem}\label{l: connected intersection}
Let $U_1, U_2$ be connected unipotent $K'$-subgroups of $G$. Then $U_1(\Kbar)\cap U_2(\Kbar)$ is the set of points over $\Kbar$ for a connected unipotent $K'$-subgroup of $G$.
\end{lem}
\begin{proof}
It is clear that $U_1(\Kbar)\cap U_2(\Kbar)$ is the set of points over $\Kbar$ for a unipotent $K'$-subgroup of $G$, which we denote by $U_1\cap U_2$. We need to show connectedness.

Write $\lieu_1$ (resp. $\lieu_2$) for the Lie algebra of $U_1$ (resp. $U_2$). Let $X$ be an element of $\lieu_1(\Kbar)\cap \lieu_2(\Kbar)$; then $\exp X \in (U_1\cap U_2)(\Kbar)$. Conversely if $X\not\in \lieu_1(\Kbar)\cap \lieu_2(\Kbar)$, then either $\exp X\not\in U_1(\Kbar)$ or $\exp X \not\in U_2(\Kbar)$. We conclude that $\exp(\lieu_1(\Kbar)\cap \lieu_2(\Kbar)) = (U_1\cap U_2)(\Kbar)$. Now Lem. \ref{l: kirillov} implies that $U_1\cap U_2$ is connected.
\end{proof}

Define $G^i = U^i T$; since $U_i$, $T$, and the action of $T$ on $U_i$ are defined over $K'$, we conclude that $G^i$ is also defined over $K'$, and hence is $K'$-split. We can define $(U^i)_\Lambda$ with respect to the $G_i$; then  Lem. \ref{l: desc} implies that 
$$(U^i)_\Lambda = U_\Lambda\cap U^i.$$
On the other hand we can define $(U^i)_R$ with respect to the group. Observe that $(U^i)_R\leq U^i \cap U_R$.

Let $U_L$ be the last term in the lower central series of $G$. The next lemma asserts that $U_L$ and $U_R$ are equal: this will be important later as it implies that $G(K)/U_R(K)$ is nilpotent.

\begin{lem}\label{l: urul}
$U_R=U_L$ and $G=U_RE$, where $E=C_G(T)$, a Cartan subgroup of $G$.
\end{lem}
\begin{proof}
By definition $G/U_L$ is nilpotent and so, by \cite[10,6]{borel}, $G/U_L\cong
U/U_L \times T$. Let $R$ be a root subgroup; then $[R,T]\neq \{1\}$.

Now suppose that $R\cap U_L=\{1\}$. Then $G/U_L$ contains a normal subgroup that does not commute with $T$. This is a contradiction. 

Thus $R\cap U_L$ is non-trivial. Since $R$ is 1-dimensional, Lem. \ref{l: connected intersection} implies that $R<U_L$. We conclude that all root subgroups lie in $U_L$ and, in particular, $U_L$ contains $U_R$.

Conversely we want to prove that $U_R$ contains $U_L$; equivalently we can show that $G/U_R$ is nilpotent. Since $E$ is nilpotent, it is sufficent to prove that $G=U_RE$; equivalently, we show that $U=U_RU_\Lambda$. 

This is immediate if $U$ is abelian. Now suppose that the result holds for $U$ of nilpotency rank less than $s$. Write $u=x_{R_1}(s_1) \cdots x_{R_k}(s_k)$. Observe that
$$u U^1 = \prod\limits_{R_i\in(\Phi_R^*)^1} x_{R_i}(s_i')\prod\limits_{R_j\in\Lambda^1}x_{R_j}(s_j') U^1.$$
for some $s_i', s_j'\in \Kbar$. By induction we can write $U^1 = (U^1)_R(U^1)_\Lambda$. Thus we can write
$$u = \prod\limits_{R_i\in(\Phi_R^*)^1} x_{R_i}(s_i') \prod\limits_{R_j\in\Lambda^1}x_{R_j}(s_j') v_R v_\Lambda,$$
where $v_R\in (U^1)_R$ and $v_\Lambda\in(U^1)_\Lambda$. Now Lem.\ \ref{l: conj
  lambda} implies that, for $R_j\in\Lambda^1$, and $R_i$ a root subgroup in $U^1$,
the group $x_{R_j}(s_j') R_i (x_{R_j}(s_j'))^{-1}$ is a root subgroup in $U^1$, and so must lie in $(U^1)_R$. Thus, in particular,
$$u = \left(\prod\limits_{R_i\in(\Phi_R^*)^1}x_{R_i}(s_i') v_R'\right) \left(\prod\limits_{R_j\in\Lambda^1}x_{R_j}(s_j') v_\Lambda\right),$$ 
for some $v_R'\in (U^1)_R$. But now observe that 
$$\prod\limits_{R_i\in(\Phi_R^*)^1} x_{R_i}(s_i')v_R'\in U_R, \textrm{ and }\prod\limits_{R_j\in\Lambda^1}x_{R_j}(s_j') v_\Lambda\in U_\Lambda;$$ 
the result follows.
\end{proof}

The above result should be compared with \cite[9.7]{boreltits2}. We have seen already that $U_L$ is defined over $K'$; hence $U_R$ is also. In particular $U_R$ is $K'$-split.



\subsection{Commutators}
For $A,B$ two $K'$-subgroups of $G$, define
$$M=\langle [a,b] \, \mid \, a\in A(\Kbar), b\in B(\Kbar) \rangle.$$
If $A$ is connected, then \cite[2.3]{borel} implies that the abstract group $M$ is in fact the set of points over $\Kbar$ for a $K'$-subgroup of $G$; we denote this $K'$-group $[A,B]$. We investigate the group $[A,B]$ for $A,B$ weight subgroups of $G$.

\begin{lem}\label{l: alex}
 Let $A,B$ be connected closed 1-dimensional subgroups of $G$ such that $[A,B]$ is central and non-trivial in $G$. Then $[A,B]$ is a 1-dimensional $K'$-subgroup of $G$. Furthermore, for a field $K$,
$$[A,B](K) = \{[a,b] \, \mid \, a\in A(K), b\in B(K) \}.$$
\end{lem}
\begin{proof}
Write $\liea$ (resp. $\lieb$) for the Lie algebra of $A$ (resp. $B$); let $H=[A,B]$ and write $\lieh$ for the Lie algebra of $H$; these are Lie subalgebras of $\lieg$, the Lie algebra of $G$, which is in turn a Lie subalgebra of $\liegl_r$.

Note that, since $H$ is central in $G$, $\lieh$ is central in $\lieg$. Now $A$ (resp. $B$) is the image of $\liea$ (resp. $\lieb$) under the $\exp$ map. Take $a\in\liea(\Kbar)$, $b\in\lieb(\Kbar)$ and consider
\begin{equation*}
\begin{aligned}[]
&[\exp(a), \exp(b)] \\
&= \exp(a)\cdot \exp(b)\cdot \exp(-a) \cdot \exp(-b) \\
&= (1+a+\frac{a^2}2+\cdots)(1+b+\frac{b^2}2+\cdots)(1-a+\frac{a^2}2+\cdots)(1-b+\frac{b^2}2+\cdots) \\
&= 1+[a,b]+ \cdots
\end{aligned}
\end{equation*}
Note that we are using $[\,, \,]$ in two ways here - as a commutator in the group, and as the Lie bracket. Note too that $1$ is the identity matrix in $\liegl_r$. Finally note that, in the last line, $1+[a,b]+\cdots$ means $1+[a,b]$ {\it plus higher order Lie brackets}. Since $\lieh$ is central in $\lieg$ we conclude that
$$[\exp(a),\exp(b)] = 1+[a,b].$$

Now observe that for $k,l\in \Kbar$
$$(1+[a,b])(1+[ka, lb]) = 1+(kl+1)[a,b]+ \cdots$$
Again we can ignore the higher order Lie brackets. In particular this implies that the set of commutators 
$$\{[u,v] \, \mid \, u\in A(\Kbar), v\in B(\Kbar)\}$$
is a group, and so is equal to $[A,B](\Kbar)$. Moreover, for fixed 
$a\in\liea(\overline{K})$, $b\in\lieb(\overline{K})$, this group is equal to
$$\{1+k[a,b] \, \mid \, k\in\Kbar\}$$
Clearly the map
$$G_a \to [A,B], \, k\mapsto 1+k[a,b]$$
is a morphism of algebraic groups, and we conclude that $[A,B]$ is
one-dimensional as required. If $A$ and $B$ are defined over $K'$, then
$a$, $b$ can be chosen to be in $\mathfrak{a}(K')$ and
$\mathfrak{b}(K')$, respectively, and so $\lbrack A,B\rbrack$ is defined over
$K'$.
\end{proof}

Note that the Baker-Campbell-Hausdorff formula yields an alternative proof of Lem. \ref{l: alex}.

\begin{cor}
 Suppose that $A,B$ are weight subgroups of $G$ such that $[A,B]$ is non-trivial and central in $U$. Then $[A,B]$ is a weight subgroup of $G$.
\end{cor}
\begin{proof}
The previous lemma implies that $[A,B]$ is the set of commutators of $A$ and $B$. Now take $u\in A(\Kbar), v\in B(\Kbar), t\in T(\Kbar)$. Observe that
$$t[u,v]t^{-1} = [tut^{-1},tvt^{-1}].$$
Since $A$ and $B$ are weight groups, $T$ normalizes $A$ and $B$ and we conclude that $t[u,v]t^{-1}\in[A,B](\Kbar)$ as required.
\end{proof}

\begin{lem}\label{l: interesting 1}
 Either $G$ is nilpotent, or $(\Phi^*)^1$ is non-empty.
\end{lem}
\begin{proof}
 Suppose that $(\Phi^*)^1$ is empty; in other words $\alpha(R) = 1$ for all $R\in\Phi^1$. Since $U(\Kbar)$ is generated by $\{R(\Kbar) \, \mid \, R\in\Phi^1\}$, this implies that $U$ is centralized by $T$. So $G=U\times T$ and \cite[10.6]{borel} implies the result.
\end{proof}

\subsection{Root kernels}

Recall that the action of $T$ on a root subgroup $R$ induces a character $\alpha: T\to \GL_1$. We note first of all that this character (which we call a root) is a regular map over $K'$.

Now given such a root $\alpha: T\to \GL_1$ we can extend to a character $\alpha: G\to\Kbar$ simply by defining $\alpha(g)=\alpha(t)$ where $g=ut$ for $u\in U$, $t\in T$.

In what follows the {\it kernel} of a root will be important; to ensure that there is no confusion we write $\ker_G(\alpha)$ (resp. $\ker_T(\alpha)$) when we want to think of $\alpha$ as a function from $G$ (resp. $T$) to $\Kbar$. Note that the group $\ker_G(\alpha)$ is a solvable linear algebraic group defined over $K'$. 

We will require that root kernels are connected; this fact is not true in general. However if we restrict the structure of the group $G$, then this fact holds. We clarify how we make this restriction in the following lemma.

\begin{lem}\label{l: torus normal}
Let $U_0$ be a unipotent $K'$-subgroup of a Borel subgroup $B=U_rT_r$ of $\GL_r$, with $B$ also defined over $K'$. Then $N_{T_r(\Kbar)}(U_0(\Kbar))$ is the set of points over $\Kbar$ of $T_0$, a connected $K'$-subgroup of $T$.
\end{lem}

Note that a connected $K'$-subgroup of $T$ is, precisely, a subtorus of $T$.

\begin{proof}
Write $T_r(\Kbar)$ as the set of invertible diagonal matrices. Let $\Phi_R=\{R_1, \cdots, R_d\}=\Phi_R^*$ be a set of weight groups for the group $B=UT$; let $\phi_i:T(\Kbar)\to\Kbar$ be the root associated with $R_i$ for $i=1,\dots, d$. Let $\mathfrak{r}_i$ be an element of $\lieu(\Kbar)$ such that $\exp(\mathfrak{r}_i)\in R_i(\Kbar)$; then $\{\mathfrak{r}_1, \dots, \mathfrak{r}_d\}$ is a basis for $\lieu(\Kbar)$.

Now write $N$ for $N_{T_r(\Kbar)}(U_0(\Kbar)))$ and observe that $N$ is a subgroup of $T(\Kbar)$; one can therefore apply Lem. \ref{l: roots} to the group $U_0\rtimes N$. (Although Lem. \ref{l: roots} is stated for a closed connected solvable group $G$; the proof follows through for any simultaneously diagonalizable abstract group (such as $N$), diagonalizing a closed unipotent group (such as $U_1)$.) Write $E_1, \dots, E_{d_1}$ for the resulting set of weight subgroups in $U_1$; choose $\mathfrak{e}_i\in \lieu(\Kbar)$ such that $\exp(\mathfrak{e}_1)\in E_i(\Kbar)$ for $i=1,\dots, d_1$.


The condition that $E_l$ is a weight subgroup can now be translated into a statement about the expansion of vector $\mathfrak{e}_l$ in terms of the basis $\{\mathfrak{r}_1, \dots, \mathfrak{r}_d\}$. Write 
$$\mathfrak{e}_l = a_1\mathfrak{r}_1+ \cdots +  a_d \mathfrak{r}_d$$
for $a_1,\dots, a_d\in\Kbar$. Define
$$\Phi_R^{\mathfrak{e}_l}=\{ R_i\in\Phi_R \, \mid \, a_i\neq 0\}.$$
Then $E_l$ is a weight subgroup if and only if for all $g\in N$, for all $R_i,R_j\in {\Phi_R^{\mathfrak{e}_l}}$, we have 
\begin{equation}\label{e: equality}
\phi_i(g)=\phi_j(g).
\end{equation}

Thus the group $N$ satisfies a number of equations of the form (\ref{e: equality}) for various $i,j\in \{1,\dots, n\}$. Conversely, these equations define a closed, connected $K'$-subgroup $T_0$ of $T_r$ such that $T_0(\Kbar)$ normalizes $U_1(\Kbar)$. We conclude, therefore, that $N=T_0(\Kbar)$ as required. 

\end{proof}

\begin{cor}
The roots $T_0\to \GL_1$ with respect to the group $U_0T_0$ are restrictions of the roots $T_r\to \GL_1$ with respect to the group $U_rT_r$.
\end{cor}
\begin{proof}
Using the notation of the previous proof it is clear that $\alpha(E_i)=\alpha(R_j)$ where $R_j\in\Phi_R^{\mathfrak{e}_i}$.
\end{proof}

\begin{cor}\label{c: torus normal}
Let $\xi_1,\dots, \xi_m: T_0\to \GL_1$ be a subset of a set of roots with respect to the group $U_0T_0$. Then the group
\begin{equation*}
T_m=\ker_{T_0}(\xi_1) \cap \cdots \cap \ker_{T_0}(\xi_m)
\end{equation*}
is a subtorus of $T_0$. Furthermore if $m\geq 1$, then $\dim T_m < \dim T_0$ and if $\xi_1,\dots, \xi_m$ are all the roots with respect to the group $U_0T_0$, then $U_0T_m$ is nilpotent.
\end{cor}
\begin{proof}
The previous corollary implies that $\xi_1,\dots, \xi_m$ can be extended to roots $T_r\to\GL_1$ with respect to the group $U_r T_r$. Let $\phi_1,\dots, \phi_d:T_r\to \GL_1$ be a full set of roots for the group $B=U_rT_r$; then the group $T_m$ is defined by a finite set of equations of the form
$$\phi_i=\phi_j, \, \, \, \, \phi_l=1,$$
for various choices of $i,j$ and $l$. Clearly these equations define a subtorus of $T$ as required.

If $m\geq1$, then $T_m$ is a proper subgroup of $T_0$; then, since $T_0$ is connected, we have $\dim T_m<\dim T_0$. Finally, if $\xi_1,\dots, \xi_m$ are all the roots with respect to the group $U_0T_0$, then $T_m$ centralizes $U_0$, and so $U_0T_m = U_0\times T_m$ is nilpotent as required.
\end{proof}

\section{From abstract solvable groups to linear algebraic solvable groups}\label{s: reduction}

In order to prove Thm. \ref{t: main} we need to establish the connection (in the context of growth) between abstract solvable subgroups of $\GL_r(K)$ and connected solvable linear algebraic subgroups of exponential type in $\GL_r$ that are defined over a finite field $K'$. Establishing this connection is the aim of this section; specifically we prove the following result:

\begin{prop}\label{p: bounded degree}
Let $G$ be a subgroup of $\GL_r(K)$.
Let $H<G$ be a subgroup of finite index.

Suppose that, for every finite subset $A\subset H$ and every
$C\geq 1$ there is an integer $k\ll_{r} 1$ such that either
\begin{enumerate}
\item\label{eq:gaca1}
$|A_3|\geq C |A|$, or else
\item\label{eq:gaca2}
$\langle A \rangle$ contains a subgroup $U_R$ and a normal
subgroup $S$ such that
\begin{itemize}
\item $U_R$ is unipotent and $S$ is solvable,
\item $U_R \lhd S$ and $S/U_R$  is nilpotent,
\item $A_k$ contains $U_R$, and
\item $A$ is contained in the union of at most $C^{O_r(1)}$ cosets of $S$.
\end{itemize}
\end{enumerate}

Then, for every finite subset $A\subset G$ and every $C\geq 1$, we have the same
conclusion: either (\ref{eq:gaca1}) holds or (\ref{eq:gaca2}) holds
(with $O_r(1)$ replaced by  $O_{r,|G:H|}(1)$).
\end{prop}

We remark that if we add the requirement that $\langle A\rangle = H$ to the conditions, then
we obtain the conclusion above provided the set $A$ satisfies the condition that $\langle A \rangle = G$.
(This is so because, in Prop.\ \ref{prop:usef},
 (\ref{eq:gagar}) gives us that $\langle A\rangle = G$
implies $\langle A_H\rangle = H$.)

To prove Prop.~\ref{p: bounded degree} we will need, first, a classical result of Mal'cev \cite{malcev} (see also \cite[(3.1.6)]{lr}) concerning the structure of solvable subgroups of $\GL_r(F)$ where $F$ is an algebraically closed field.

\begin{prop}\label{p: break down}
Let $S$ be an abstract solvable subgroup of $\GL_r(F)$ where $F$ is an algebraically closed field. Then $S$ contains a subgroup $H$ such that $[S:H]\ll_r1$ and $H$ is trigonalizable over $F$.
\end{prop}

\begin{prop}\label{p: bounded borel}
Let $S$ be an abstract solvable subgroup of $\GL_r(K)$. Then $S$ has a normal subgroup $H$ such that $[S:H]\ll_r1$, and $H$ lies in $B(K)$ where $B$ is a Borel subgroup of $\GL_r$ defined and trigonalizable over $K'$, a field extension of $K$ of degree at most $r$. 
\end{prop}
\begin{proof}
Observe first that if $S$ admits a subgroup $H$ satisfying all conditions except for normality, then we are done (we simply take the core of $H$ - the intersection of its conjugates in $S$ - to be the normal subgroup we are looking for). This observation and Prop.~\ref{p: break down} imply that it is sufficient to prove the following: if $H$ is a subgroup of $\GL_r(K)$ that is trigonalizable over $\overline{K}$, then $H$ lies in $B(K)$ where $B$ is a Borel subgroup of $\GL_r$ defined and trigonalizable over $K'$, a field extension of $K$ of degree at most $r$. 

The result is trivial for $r=1$ since $\GL_r(K)=B(K)$ in this case. Assume then that $r>1$. Suppose first that $H$ contains no non-trivial unipotent elements. Then $H$ lies inside a maximal torus $T$ of $\GL_r(K)$ and the result follows immediately from the standard classification of maximal tori in $\GL_r(K)$ (see, for instance, chapter 3 of \cite{carter}).

If, on the other hand, $H$ contains a unipotent element, then, in particular, $H$ contains a normal unipotent subgroup. Now the Borel-Tits theorem (\cite{boreltits}; see also \cite[Theorem 3.1.3]{gls3}) implies that $H$ lies inside a proper parabolic subgroup $P$ of $\GL_r(K)$. Since $P$ is conjugate in $\GL_r(K)$ to a group of block-diagonal matrices and since the Levi complement of $P$ is isomorphic to a direct product $\GL_{r_1}(K)\times \GL_{r-r_1}(K)$ for some $r_1>1$, the result follows by induction on $r$.
\end{proof}

The next set of results are designed to show that ``if we have growth in a subgroup of bounded index, then we have growth in the group.''

\begin{prop}\label{prop:usef}
Let $G$ be a group. Let $H\triangleleft G$ be a normal subgroup of finite index.
Let $A\subset G$ such that $\langle A\rangle=G$.

Then there is a subset $A_H\subset A_k \cap H$, $k\ll_{|G:H|} 1$,
 such that
\begin{equation}\label{eq:msta}
A \subset \bigcup_{g\in J} g A_H,
\end{equation}
\begin{equation}\label{eq:gagar}
\langle A\rangle = \bigcup_{g\in J} g\langle A_H\rangle,\end{equation}
where $J\subset A_k$ 
is a subset of a full set of coset representatives of $G/H$, and
$\langle A_H\rangle$ is normal in $\langle A\rangle$.

Moreover, $|A|\ll_{|G:H|} |A_H|\ll_{|G:H|} |A|$.
Furthermore, given any $H'\triangleleft \langle A_H\rangle$,
\[\left(\bigcap_{g\in J} g H' g^{-1}\right) \triangleleft \langle A\rangle.\]
Lastly, for every $g\in J \cup J^{-1}$, $g A_H g^{-1} \subset (A_H)_3$.
\end{prop}
\begin{proof}
We can assume without loss of generality that $\{a\cdot H: a\in A\}$ generates
$G/H$. Thus, for every left coset of $H$, we can find a $g\in A_k$
($k\leq |G:H|$) contained in that coset. 
Write $A = \cup_{g\in J} g C_g$, where $J\subset A_k$ 
is a full set of coset representatives
of $G/H$ and $C_g\subset H$ for every $g\in J$. 
We can choose $J$ so that $e\in J$ and $J = J^{-1}$.

Let \begin{equation}\label{eq:shwa}
A_H = \bigcup_{g\in J_2} \bigcup_{g'\in J}
g (C_{g'} \cup C_{g'}^{-1}) g^{-1}\; \cup \bigcup_{g\in J_2} 
\left\{\overline{g^{-1}} g, g \overline{g^{-1}}\right\},\end{equation}
where, for $g\in G$, $\overline{g}$ denotes the element of $J$ in the same left coset of
$H$ as $g$.
Since $H\triangleleft G$ and $C_g\subset H$ for every $g\in J$, 
$A_H$ is contained in $H$. It is clear that
$A_H\subset A_{k'} \cap H$ with $k' = 5 k + 1$. It is also clear that
$|A|\ll_{|G:H|} |A_H|\ll_{|G:H|} |A|$.
We also have (\ref{eq:msta}) because $A = \cup_{g \in J} g C_g$ and
$C_g \subset A_H$ for every $g\in J$ (by definition (\ref{eq:shwa})).

Let us now check that $g A_H g^{-1} \subset (A_H)_3$ for every $g\in J =
J\cup
J^{-1}$. Let $a\in A_H$. If $a\in g_0 (C_{g'} \cup C_{g'}^{-1}) g_0^{-1}$ for
some $g_0\in J_2$, then $g a g^{-1} \in g_2 g_1 (C_{g'} \cup C_{g'}^{-1})
g_1^{-1} g_2^{-1}$ for some $g_1 \in J\cup J^{-1} \cup \{e\}$,
$g_2\in J_2$. Let $g_3 = \overline{g_2^{-1}}^{-1} \in J^{-1}$. Then
\[g a g^{-1} \in g_2 \overline{g_2^{-1}} \cdot g_3 g_1 (C_{g'} \cup
C_{g'}^{-1}) g_1^{-1} g_3^{-1} \cdot (g_2 \overline{g_2^{-1}})^{-1} \in
(A_H)_3,\]
as was desired.

It remains to show that $\langle A\rangle
= \bigcup_{g\in J} g\langle A_H\rangle$. The inclusion 
$g\langle A_H\rangle \subset \langle A\rangle$, $g\in J$, is easy.
To show that $\langle A\rangle = \langle \bigcup_g g C_g\rangle$ is contained
in $\bigcup_{g\in J} g\langle A_H\rangle$, it is enough to show that, if 
$x\in \bigcup_{g\in J} g C_g$ and $y\in \bigcup_{g\in J} g \langle A_H\rangle$,
then $x y$ and $x^{-1} y$ are in $\bigcup_{g\in J} g \langle A_H\rangle$. 

Let us see: for $x$ and $y$ as above, $x y = g c g' a$ for some
$g,g'\in J$, $c\in C_g$, $a\in \langle A_H\rangle$, and so 
\[\begin{aligned}
x y &= g c g' a = g g' (g')^{-1} c g' a \in g g' \langle A_H\rangle\\
&= \overline{g g'}\cdot \overline{(g g')^{-1}} g g' \langle A_H\rangle
=  \overline{g g'} \langle A_H\rangle = g'' \langle A_H\rangle
\end{aligned}\] 
for some $g''\in J$. Similarly,
\[\begin{aligned}
x^{-1} y &= c^{-1} g^{-1} g' a = g^{-1} g' (g^{-1} g')^{-1} c^{-1} g^{-1} g' a
\in g^{-1} g' \langle A_H\rangle\\
&= \overline{g^{-1} g'}\cdot \overline{(g^{-1} g')^{-1}} g^{-1} g'
 \langle A_H\rangle = \overline{g^{-1} g'} \langle A_H\rangle = 
g''\langle A_H\rangle
\end{aligned}\] for some $g''\in J$. Hence
$\langle A\rangle \subset \cup_{g\in J} g\langle A_H\rangle$, and so
$\langle A\rangle = \cup_{g\in J} g\langle A_H\rangle$.

To show that $\langle A_H\rangle$ is normal in $\langle A\rangle$,
it is enough to show that $g A_H g^{-1} \subset \langle A_H\rangle$
for every $g\in J\cup J^{-1}$.
 First, note that, for all $g''\in J$, $g\in J_2$,
$g'\in J\cup J^{-1}$, $c\in C_{g'} \cup C_{g'}^{-1}$,
\[g'' g c (g'' g)^{-1} = g'' g\overline{(g'' g)^{-1}} 
(\overline{(g'' g)^{-1}})^{-1} c \overline{(g'' g)^{-1}}
(\overline{(g'' g)^{-1}})^{-1} (g'' g)^{-1} \in A_H \cdot A_H \cdot A_H
\subset \langle A_H\rangle,\]
where we recall that $\overline{g}\in J$ for every $g\in G$. Next,
we see that, for $g_1,g_2,g_3\in J\cup J^{-1}$,
\[\begin{aligned}
g_1 \overline{(g_2 g_3)^{-1}} 
g_2 g_3 g_1^{-1} 
&= g_1 \overline{(g_2 g_3)^{-1}} \overline{(g_1 \overline{(g_2 g_3)^{-1}})^{-1}}
\overline{g_1 \overline{(g_2 g_3)^{-1}}}
g_2 \overline{g_3 g_1^{-1}} \overline{(g_3 g_1^{-1})^{-1}}
 g_3 g_1^{-1} \\ &\in A_H
\overline{g_1 \overline{(g_2 g_3)^{-1}}}
g_2 \overline{g_3 g_1^{-1}}
 A_H .\end{aligned}\]
Now \[
(g_2 \overline{g_3 g_1^{-1}})^{-1} = \overline{(g_3 g_1^{-1})^{-1}} g_2^{-1}
= g_1 g_3^{-1} h g_2^{-1} = g_1 g_3^{-1} g_2^{-1} h' = g_1 \overline{(g_2 g_3)^{-1}} h'' h' 
\] for some $h,h',h''\in H$. Hence
$\overline{(g_2 \overline{g_3 g_1^{-1}})^{-1}} =
\overline{g_1 \overline{(g_2 g_3)^{-1}}}$, and so
\[\overline{g_1 \overline{(g_2 g_3)^{-1}}}
g_2 \overline{g_3 g_1^{-1}}\in A_H.\]
We conclude that $g_1 \overline{(g_2 g_3)^{-1}} 
g_2 g_3 g_1^{-1} \in \langle A_H\rangle$. By a similar argument,
$g_1 g_2 g_3 \overline{(g_2 g_3)^{-1}} 
g_1^{-1} \in \langle A_H\rangle$. Hence 
$g_1 A_H g_1^{-1} \subset \langle A_H\rangle$
for every $g_1\in J\cup J^{-1}$, as desired.

Let us now examine
$H'' = \bigcap_{g\in J} g H' g^{-1}$, where $H'\triangleleft \langle A_H\rangle$.
 For $g, g'\in J$, $h\in \langle A_H\rangle$,
\begin{equation*}
\begin{aligned}
g' h g H' g^{-1} h^{-1} (g')^{-1} &= g' g g^{-1} h g H' g^{-1} h^{-1} g g^{-1} (g')^{-1}\\
&= g' g h' H' (h')^{-1} (g' g)^{-1} \\
&= g' g H' (g' g)^{-1}, 
\end{aligned}
\end{equation*}

where $h' = g^{-1} h g \in \langle A_H\rangle$. 
(Recall that $\langle A_H\rangle$ is normal in $\langle A\rangle$.)
Thus
\[\begin{aligned}
g' h H'' (g' h)^{-1} &= 
\bigcap_{g\in J} g' h g H' g^{-1} h^{-1} (g')^{-1} =
\bigcap_{g\in J} g' g H' (g' g)^{-1}\\  &= \bigcap_{g\in J} \overline{g' g}
\overline{(g' g)^{-1}} g' g 
H' (\overline{(g' g)^{-1}} g' g)^{-1} \overline{g' g}^{-1}.\end{aligned}\]
Now $\overline{(g' g)^{-1}} g' g \in A_H$, and thus normalises $H'$. 
As $g$ runs through the elements of $J$ while $g'$ is fixed, $\overline{g' g}$
runs through each element of $J$ exactly once. Hence
\[g' h H'' (g' h)^{-1} = \bigcap_{g\in J} g H' g^{-1} = H''\]
for all $g\in J$, $h\in \langle A_H\rangle$, and so $g H'' g^{-1} = H''$
for all $g\in \langle A\rangle$, as was desired.
\end{proof}

The following lemma is basic.

\begin{lem}\label{lem:vangeli}
Let $H$ be a group. Let $H_1\triangleleft H$, $H'<H$. Then
$(H_1 \cap H') \triangleleft H'$. Moreover, 
$H'/(H_1 \cap H')$ is isomorphic to a subgroup of
$H/H_1$.
\end{lem}
\begin{proof}
For any $g\in H'$ and any $h\in H_1\cap H'$, we have $g h g^{-1}\in H_1$
(because $H_1$ is normal) and $g h g^{-1} \in H'$ (because $g$ and $h$
are in $H'$). Thus, $H_1 \cap H' \triangleleft H'$.

We define a map $\iota:H'/(H_1\cap H')\to H/H_1$ as follows:
$\iota(g (H_1\cap H')) = g H_1$. It is easy to see that the map is
a well-defined homomorphism. Since its kernel is $\{e\}$, it is also
injective. 
\end{proof}

The following is a slight generalisation of \cite[Lem. 7.16]{helfgott3}.

\begin{lem}\label{lem:coc}
Let $M$ be a group. Let $N_1, N_2,\dotsc, N_k \triangleleft M$. 
Let $A\subset M$ be such that
$A$ is contained in the union of $\leq n_j$ left cosets of $N_j$ for $j=1,2,\dotsc,k$.
Then $A$ is contained in the union of $\leq n_1 n_2\dotsb n_k$ 
left cosets of $N_1\cap N_2 \cap \dotsc \cap N_k$.
\end{lem}
\begin{proof}
The map $\iota:M/(N_1 \cap N_2 \cap \dotsc \cap N_k) 
\to M/N_1 \times M/N_2 \times \dotsb \times M/N_k$ given by
$\iota(g (N_1 \cap N_2 \cap \dotsc \cap N_k)) 
= (g N_1, g N_2,\dotsc , g N_k)$ is a well-defined homomorphism;
since its kernel is trivial, it is also injective. The image of
$\iota(A\cdot (N_1 \cap N_2\cap \dotsc \cap N_k)$ 
is of size at most $n_1 \cdot n_2 \dotsb n_k$;
hence $A\cdot (N_1\cap N_2 \cap \dotsc \cap N_k) 
\subset M/(N_1 \cap N_2 \cap \dotsc \cap N_k)$ is of size at most
$n_1 \cdot n_2 \dotsb n_k$.
\end{proof}

We are now able to prove the main result of this section.

\begin{proof}[Proof of Prop.~\ref{p: bounded degree}]
It is well-known that a subgroup of a group $G$ of index $m$ always contains
a normal subgroup of $G$ of index $\leq m!$ (take the kernel of the 
representation of $G$ by left multiplication on $G/H$). Thus, we may assume
without loss of generality that $H$ is normal in $G$.

Let $A\subset G$ and $C\geq 1$ be given. Suppose that $|A_3|\leq 2|A|$; then Lem. \ref{l: olson} implies that $A_3 = \langle A \rangle$ and (\ref{eq:gaca2}) follows immediately with $U_R=S=\langle A \rangle$. So assume that $C\geq2$.
 
Let $A_H$ and $J$ be as in Prop. \ref{prop:usef}.
Suppose conclusion 
(\ref{eq:gaca1}) in the statement of the present proposition
 does not hold for $A_H$, as otherwise (\ref{eq:gaca1}) 
for $A$ follows immediately. Then conclusion (\ref{eq:gaca2}) must hold
for $A_H$; denote the subgroups we obtain by $U_{R,H}$ and $S_H$.

Let $S = \bigcap_{g\in J} g S_H g^{-1}$. By Prop.\ \ref{prop:usef}
(with $H' = S_H$), we have
 $S\triangleleft \langle A\rangle$. Let $U_R = S\cap U_{R,H}$.
By Lem. \ref{lem:vangeli}, $U_R$ is a normal subgroup of $S$ and $S/U_R$
is isomorphic to a subgroup of $S_H/U_{R,H}$. Hence $S/U_R$ is nilpotent.
Since $(A_H)_k$ contains $U_{R,H}$, it is obvious that $A_k$ (which contains
$(A_H)_k)$) contains $U_{R}\subset U_{R,H}$. 

It remains to bound the number of  cosets occupied by $A$. We are given that
$A_H$ lies in at most $C^{O_r(1)}$  cosets of $S_H$. 
By Prop.\ \ref{prop:usef}, $g^{-1} A_H g \in (A_H)_3$ for
every $g\in J$. Hence $g^{-1} A_H g$ lies in at most $C^{3 O_r(1)}$ left cosets of $S_H$.
(Recall that $S_H\triangleleft \langle A_H\rangle$.) Thus
$A_H$ lies in at most $C^{3 O_r(1)}$  cosets of $g S_H g^{-1}$.
Therefore, by Lem. \ref{lem:coc}, $A_H$ is contained
in at most $C^{3 |J| O_r(1)}
\leq C^{3 |G:H|}$ cosets of 
$S = \bigcap_{g\in J} g S_H g^{-1}$. Thus, by (\ref{eq:msta}),
 $A$ is contained in at most
\[|J| C^{O_r(|G:H|)}
\leq |G:H| C^{O_r(|G:H|)}\leq C^{O_{r, |G:H|}(1)}\] cosets of $S$.
\end{proof}

\section{Growth when $U$ is abelian}\label{s: first interesting}

As we shall see when we come to prove Thm. \ref{t: main} in Section \ref{s: proof}, the results of the previous section allow us to work under some extra assumptions.

For this section we let $A_0$ be a set contained in $G_0(K)$,
where $G_0$ is a connected solvable linear algebraic subgroup of $\GL_r$ that is defined, and trigonalizable, over a finite extension $K'/K$. We require, in addition, that $G_0$ is of exponential type in $\GL_r$.

We write $G_0=U_0T_0$. We assume that
\begin{equation}\label{e: sz}
\langle A_0 \rangle = \left(\langle A_0 \rangle \cap U_0(K) \right) \rtimes \left(\langle A_0 \rangle \cap T_0(K)\right).
\end{equation}
We are able to do this since $G_0(K) = U_0(K)\rtimes T_0(K)$; then the Schur-Zassenhaus theorem implies that there exists $g\in G_0(K)$ such that $\langle a^g \, \mid \, a\in A_0\rangle$ satisfies (\ref{e: sz}). We can then study the set $\{a^g \, \mid \, a\in A_0 \}$ in order to establish all the results we need concerning $A_0$.

Our focus for this section is on the group $G = G_0/(U_0)^1$. Write $G=UT$, and observe that $U$ is abelian. Define $\Phi, \Phi^*, \Phi_R = \{R_1, \dots, R_d\}, \Phi_R^*, \Lambda, U_R$, and $U_\Lambda$ as per Section \ref{s: background solvable}. Write $A$ for the set $A_0/(U_0)^1(K)$; thus $A$ is a subset of $G(K)$.

Let us note two easy consequences \cite[9.7]{boreltits2} of the fact that $U$ is abelian:
\begin{equation}\label{e: abelian}
\begin{aligned}
U&=U_\Lambda \times U_R; \\
[G,T]&=U_R.
\end{aligned}
\end{equation}
In fact, we can do a little better:
\begin{lem}\label{l: coincide}
Assume $U$ is abelian. Then $$[G,G]=U_R.$$
\end{lem}
\begin{proof}
In light of the fact that $[G,T]=U_R$ it is sufficient to prove that $[G,G]\leq U_R$. Take $g,h\in G(\Kbar)$ and write these in standard form:
$$g= x_{R_1}(s_1)\cdots x_{R_d}(s_d)t, \, \, 
h = x_{R_1}(s_1')\cdots x_{R_d}(s_d')t'.$$
Then observe that, since $U$ and $T$ are abelian,
\begin{equation*}
\begin{aligned}[] 
[g,h]&= ghg^{-1}h^{-1} \\
&= x_{R_1}(s_1)(x_{R_1}(-s_1))^{t'}(x_1(s_1'))^tx_{R_1}(-s_1') \cdots
x_{R_d}(s_d)(x_{R_d}(-s_d))^{t'}(x_d(s_d'))^tx_{R_d}(-s_d').
\end{aligned}
\end{equation*}
If $R\in\Lambda$ then the action of $T$ on $R$ is trivial. Thus we obtain
$$[g,h] = \prod_{R\in \Phi_R^*} x_R(t_R)$$
for some $t_R\in\Kbar$. Clearly $[g,h]\in U_R(\Kbar)$.
\end{proof}

%

%
%

\begin{lem}\label{lem:pqmg}
Assume $U$ is abelian. Let $g\in G(K)$ lie outside the kernel of every root.
Then
\[\phi_g:x \to \lbrack g,x\rbrack\]
is an injective map from $U_R(K)$ to $U_R(K)$.
\end{lem}
\begin{proof}
By Lem. \ref{l: coincide},
$\phi_g(U_R)\subset U_R$. Now suppose that $g x g^{-1} x^{-1} = g y g^{-1} y^{-1}$
for $x,y\in U_R(K)$, $x\ne y$. Then $g^{-1} x^{-1} y g = x^{-1} y$, i.e.,
$g$ has a fixed point in $U_R(K)$ other than the identity. For $U$ abelian,
this contradicts
the assumption that $g$ lie outside the kernel of every root.
\end{proof}

\begin{prop}\label{p: newinterestgrow}
Let $K,A,$ and $G$ be as defined at the start of this section. There exists a positive integer $k\ll_{r} 1$ such that, for $C\geq 1$, one of the following holds:
\begin{enumerate}
\item\label{it:jola}
 $|A_k\cap \ker_G(\alpha(R))(K)|\geq  \frac1C|A|$ for some $R\in\Phi_R^*$;
\item\label{it:jolb} $|A_k|\geq C|A|$;
\item\label{it:jolc}  $A_k$ contains a normal subgroup $H$ of $U(K)$ such that $\langle A\rangle/H$ is abelian.
\end{enumerate}
\end{prop}
\begin{proof}
We apply Lem. \ref{l: b3} to the set $A$ with $G=G(K)$, $N=U(K)$, and 
\[R=\bigcup\limits_{R\in\Phi_R^*}\ker_G(\alpha(R)))(K) .\] We obtain that either
\begin{equation}\label{e: 1} 
|A/U(K)\cap \bigcup\limits_{R\in\Phi_R^*}\ker_G(\alpha(R))(K)/U(K)|\leq 
\frac1C|A/U(K)|
\end{equation}
or $|A_3\cap \cup_{R\in\Phi_R^*}\ker_G(\alpha(R))(K)/U(K)|\geq \frac1C|A|$.
The latter option implies (\ref{it:jola}). Assume, instead, that
 (\ref{e: 1}) holds.

Apply Prop.\ \ref{prop:generous} with $G=U_R(K)$, $\Gamma = G(K)/U(K)$,
$X = A/U(K)$ and $W = \lbrack A, A_2 \rbrack$.
(Note that, by Lem. \ref{l: coincide}, $W\subset U_R(K)$.)

Suppose first that conclusion (\ref{eq:hort1}) holds. Then
\begin{equation}\label{eq:lefev}
|A_{48}\cap U(K)| \geq C|\lbrack A, A_2\rbrack|,\end{equation}
where we are using (\ref{e: 1}) and the fact that an element not in the kernel
of any root acts without fixed points on $U_R(K)$ (for $U$ abelian).
Now, by (\ref{e: 1}), $A$ contains at least one element $g$ not in the kernel
of any root. By Lem. \ref{lem:pqmg}, this implies that $|\lbrack A,A_2\rbrack|
\geq |\lbrack g, A_2 \cap U_R(K)\rbrack| \geq |A_2\cap U_R(K)|$. Hence, by 
(\ref{eq:lefev}) and Lem. \ref{l: b1},
\[|A_{49}| \geq C|A|,\]
and so (\ref{it:jolb}) holds.

Suppose now that conclusion (\ref{eq:hort2}) holds. Then $A_{56}$ contains
a subgroup $V$ of $U_R(K)$ containing $\lbrack A,A\rbrack$. This subgroup is normal in $\langle A \rangle$ since $U(K)$ is abelian and by construction $V$ is normalized by $\langle A \rangle/U(K)$. Clearly,
for any $a,a'\in A$, the images $a \mod V$ and $a'\mod V$ commute. 
Hence $\langle A\rangle/V$ is abelian, and thus (\ref{it:jolc}) holds.
\end{proof}

\section{Descent}\label{s: descent}
 
In this section we investigate what happens when possibility (c) of Prop.\ \ref{p: newinterestgrow} holds. The results of this section apply only in the specific situation when $K=\Z/p\Z$. We begin with some background results.

\begin{lem}\label{l: bch}
 let $\lieu_1$ be an ideal of a unipotent Lie algebra $\lieu$ of nilpotency class $r$, defined over a field of characteristic $p>r$. For all $u_1\in\lieu_1, u\in\lieu$ there exists $u_1'\in \lieu_1$ such that
\begin{equation}\label{bch}
u+u_1 = u+u_1' + \frac12[u, u_1'] + \frac1{12}[u,[u,u_1']] - \frac1{12}[u_1', [u, u_1']] +\cdots.
\end{equation}
\end{lem}
The right hand side of (\ref{bch}) corresonds to the Baker-Campbell-Hausdorff formula which, since $\lieu$ is nilpotent, is a finite sum. The formula is well-defined by virtue of the fact that $p>r$.
\begin{proof}
If $\lieu$ is abelian the the result is trivial. We proceed by induction on the nilpotency class of $\lieu$: suppose that the result is true for Lie algebras of nilpotency class $\leq r-1$. We apply the inductive hypothesis to $\lieu/Z(\lieu)$ which is of class $\leq r-1$; then we can find $u_1'$ such that
$$u+u_1+z = u+u_1' + \frac12[u, u_1'] + \frac1{12}[u,[u,u_1']] - \frac1{12}[u_1', [u, u_1']] +\cdots$$
for some $z\in Z(\lieu)$. But now replace $u_1'$ by $u_1'-z$ and we obtain (\ref{bch}) as required. 
\end{proof}

\begin{lem}\label{l: container}
 Let $H\leq U_r(K)$, where $K=\Z/p\Z$ and $U_r$ is a maximal unipotent subgroup of $\GL_r$ with $r<p$. Write $H=\langle g_1, \dots, g_c\rangle$ such that, for all $e=1,\dots, c-1$, the group $\langle g_1, \dots, g_e\rangle$ is of order $p^e$ and is normal in the group $\langle g_1, \dots, g_{e+1}\rangle$ which is of order $p^{e+1}$.

Let $e_i=\log(g_i)$ for $i=1, \dots, c$ and define $\lieu$ to be the $\Kbar$-span of $\{e_1, \dots, e_c\}$ in $\lieu_r$, the Lie algebra of $U_r$. Then
\begin{enumerate}
 \item\label{lie1} $\lieu$ is a Lie algebra;
\item\label{lie2} $U=\exp(\lieu)$ is a $K$-group;
\item\label{lie3} $\lieu$ is the Lie algebra of $U$;
\item\label{lie4} $H=U(K)$.
\end{enumerate}
\end{lem}

Note that (\ref{lie2}) and (\ref{lie3}) imply that $U$ is of exponential type in $\GL_r$.

\begin{proof}
If $|H|=p$ then $H=\langle g\rangle$ and $\lieu$ is equal to the $\Kbar$-span of $e=\log(g)$. This is clearly a Lie algebra so (\ref{lie1}) follows, It is obvious that $U=\exp(\lieu)$ is a group; what is more $U$ is defined by the equations $f_i(\log X)=0$ where $f_i(T)=0$ are the set of equations defining the linear subspace $\lieu$, thus $U$ is a $K$-group and (\ref{lie2}) follows. Now since $U$ is defined by the equations $f_i(\log X)=0$, it follows easily that $f_i(T)=0$ defines the tangent space to $U$, and so this tangent space is $\lieu$, and (\ref{lie3}) follows. Now (\ref{lie4}) follows from Lem. \ref{l: bijection}.

Proceed by induction and assume that the result holds for groups of order less than $p^{c-1}$ and let $H$ have order $p^c$. Write $\lieu_1$ for the $\Kbar$-span of $\{e_1, \dots, e_{c-1}\}$, $U_1$ for the group $\exp(\lieu_1)$, $\mathfrak{e}$ for the $\Kbar$-span of $\{e_c\}$ and $E$ for the group $\exp(\mathfrak{e})$. Observe that, by assumption, for all $i=1, \dots, c-1$,
\begin{equation}\label{big one}
\begin{aligned}
g_c g_i g_c^{-1}\in U_1(\Kbar)
&\Rightarrow \exp({\rm Ad}(g_c)(e_i))\in U_1(\Kbar); \\
 &\Rightarrow {\rm Ad}(g_c)(e_i)\in\lieu_1(\Kbar); \\
&\Rightarrow {\rm Ad}(g_c)(le_i)\in\lieu_1(\Kbar), \, \, \forall l\in\Kbar; \\
&\Rightarrow {\rm Ad}(\exp e_c)(le_i)\in\lieu_1(\Kbar), \, \, \forall l\in\Kbar; \\
&\Rightarrow \exp([e_c, le_i])\in U_1(\Kbar), \, \, \forall l\in\Kbar; \\
&\Rightarrow [e_c, le_i]\in\lieu_1(\Kbar), \, \, \forall l\in\Kbar; \\
&\Rightarrow [me_c, le_i] \in \lieu_1(\Kbar), \, \, \forall l,m\in\Kbar; \\
&\Rightarrow [e, u] \in \lieu_1(\Kbar), \, \, \forall e\in \mathfrak{e}(\Kbar),u\in \lieu_1(\Kbar).
\end{aligned}
\end{equation}
It follows immediately that $\lieu$ is a Lie algebra (thereby yielding (\ref{lie1})) and $\lieu_1$ is an ideal of $\lieu$. By reversing up the equivalences in (\ref{big one}) we see that
$$ghg^{-1} \in U_1(\Kbar), \, \, \forall g\in E(\Kbar), h\in U_1(\Kbar),$$
 thus $U^*=U_1(\Kbar)E(\Kbar)$ is a group.

Now (\ref{lie2}) will follow if we can show that $U^*=U=\exp(\lieu)$. To do this we prove that the the following functions are well-defined
$$\exp: \lieu \to U^* \textrm{ and } \log:U^* \to \lieu.$$
Then (\ref{lie2}) will follow from the injectivity of $\exp$ and $\log$.

Consider $u_1e\in U^*=U_1(\Kbar)E(\Kbar)$; by assumption $u_1=\exp(v_1), e=\exp(f)$ for some $v_1\in\lieu_1, f\in \mathfrak{e}$. But now
$$\log(ue)=\log(\exp(v_1)\cdot \exp(f))= v_1+f+\frac12[v_1, f]+\frac1{12}[v_1,[v_1,f]]+\cdots$$
by the Baker-Campbell-Hausdorff formula. Since $\lieu_1$ is an ideal in $\lieu$ this implies that $\log(ue)\in \lieu$ as required.

Now for $\exp$: take $v=v_1+f$ where $v_1\in \lieu_1$ and $f\in\mathfrak{e}$. Then Lem. \ref{l: bch} implies that there exists $v_1' \in \lieu_1$ such that
\begin{equation*}
 \begin{aligned}
\exp(v_1+f)&=\exp(v_1' + f + [v_1, f]+\frac1{12}[v_1,[v_1,f]]+\cdots) \\
&=\exp(v_1')\exp(f) \in U_1(\Kbar)E(\Kbar)=U^*
 \end{aligned}
\end{equation*}
as required. Thus (\ref{lie2}) is proved. 

Just as in the abelian case (\ref{lie2}) implies that $U$ is defined by the equations $f_i(\log X)=0$ where $f_i(T)=0$ are the set of equations defining the linear subspace $\lieu$; it follows easily that $f_i(T)=0$ defines the tangent space to $U$, and so this tangent space is $\lieu$, and (\ref{lie3}) follows.

Finally Lem. \ref{l: bijection} gives (\ref{lie4}).
\end{proof}

\begin{lem}\label{l: containment}
Let $A\subseteq B(K)$, where $K=\Z/p\Z$ and
$B$ is a Borel subgroup of $\GL_r$ with $p>r$. Then there
is a connected, solvable $K'$-group $G=UT$ of exponential type in $\GL_r$, where $K'$ is a finite extension of $K$, such that
$A\subseteq G(K)$, $U$ is a $K$-group, and
$$U(K)\subseteq \langle A \rangle.$$
What is more if $\xi_1,\dots, \xi_m:T\to \GL_1$ are roots with respect to $G$, then the group
\begin{equation*}
T_m=\ker_{T}(\xi_1) \cap \cdots \cap \ker_{T}(\xi_m)
\end{equation*}
is a subtorus of $T$.
\end{lem}
\begin{proof}
Recall that $B$ is a Borel subgroup of $\GL_r$ such that $B(K)$ contains $A$; write $B = U_rT_r$ for the decomposition into unipotent part and torus. Without loss of generality we assume (\ref{e: sz}) with respect to the embedding of $A$ in $B(K)$.

Write $J$ for the group $\langle A \rangle$; define $H = J\cap U_r(K)$ and apply Lem. \ref{l: container} to $H$. We obtain a $K$-group $U$ of exponential type in $\GL_r$ such that $U(K)=H\subseteq\langle A \rangle$. 

Consider $N_{T_r(\Kbar)}(U(\Kbar))$; Lem. \ref{l: torus normal} implies that this group is the set of points over $\Kbar$ of a connected $K'$-group $T$. Now $T(\Kbar)$ clearly contains $J\cap T_r(K)$; what is more, the action of $T$ on $U$ is defined over $K'$, thus we set $G=UT$ and are done.

Now the statement concerning root kernel intersections follows from Cor. \ref{c: torus normal}.
\end{proof}

Note that, in particular, Lem. \ref{l: containment} implies that (\ref{e: sz}) 
holds (with respect to the embedding of $A$ in $G(K)$); it also implies that 
$U_R(K)\subseteq \langle A \rangle$. With this in mind we can establish the hypotheses under which we operate.

\subsection{Hypotheses}\label{s: hypotheses}

Take $A$ inside $B(K)$ where $B$ is a Borel subgroup of $\GL_r$. Let $G = UT$ be a connected solvable linear algebraic subgroup of $B$ satisfying all the properties given in Lem. \ref{l: containment}. 

Define $\Phi, \Phi_R = \{R_1, \dots, R_d\}$ (with the ordering compatible with the height function), $\Phi_R^*, \Lambda, U_R$, and $U_\Lambda$ as per Section \ref{s: background solvable}. Let $(\Phi_R^*)^j=\{S^j_1, \dots, S^j_{e_j}\}$; observe that $e_j\leq r^2$ for all $j$.

Now we can apply Prop.\ \ref{p: newinterestgrow} to the set $AU^1(K)/U^1(K)$ inside the group $G(K)/U^1(K)$; we are interested in what happens when (c) of Prop. \ref{p: newinterestgrow} holds. Thus we assume that $A$ contains a set $W^{1}$ such that $W^1/U^1(K)$ is a normal subgroup of $\langle A \rangle/U^1(K)$ such that $(\langle A \rangle/U^1(K))/(W^1/U^1(K))$ is abelian. 

Lem. \ref{l: interesting 1} implies that either $G$ is nilpotent or $(\Phi_R^*)^1$ is non-empty. We assume the latter situation; then the fact that $U(K)\subseteq \langle A \rangle$ implies that $W^{1}/U^{1}(K)$ is non-trivial and is equal to $U_R(K)/ U^1(K)$. 

We assume that $p>r$ and fix a constant $C>1$; we assume that 
\begin{equation}\label{eq:rusdat}
|A_k\cap \ker_{G}(\alpha(R_j))(K)|\leq \frac1C|A|
\end{equation} 
for all $j=1,\dots, d$, and that
\begin{equation}\label{eq:rusdat2}
|A_k|\leq C|A|
\end{equation}
for all $R_j\in\Phi_R^*$ and all $k\ll_{r} 1$. 
We reiterate that the results of this section apply only when  $|K|=\Z/p\Z$.

The idea of this section is the following: we will ``descend'' down the lower central series of the group $U$ in order to prove that, for each $j = 1,2, \dots, $ there exists $k\ll_{r} 1$ such that $A_k$ contains a set $W^j$ with $W^j/ U^{j}(K) = (\langle A \rangle \cap U_R(K))/ U^{j}(K)$. Since we are assuming that (c) of Prop. \ref{p: newinterestgrow} holds, the statement is true for $j=1$; thus, our ``base case'' is satisfied.

We should note that our terminology is a little counter-intuitive: as we ``descend'' down $U$, the height of the root groups in $U^j\backslash U^{j+1}$ is seen to increase!

\subsection{Capturing $U_R(K)$}

The result we are aiming for is Cor. \ref{c: final} which states that $A_k$ contains $U_R(K)$ for some $k\ll_{r} 1$. Our first job is to show that all we need to do is obtain the product of root subgroups at each level; this is the content of Lem. \ref{l: need root groups}.

Note that Lem. \ref{l: containment} implies that there exists a connected unipotent $K$-group $V$ such that $V(K)=U_R(K)$. Write $V=V^0>V^1>V^2>\cdots$ for the lower central series of $V$. Since $V$ is defined over $K$ we have
$$V^0(K)=U_R(K), V^1(K)=[U_R(K),U_R(K)], \dots, V^{i+1}(K) = [V^i(K), V^0(K)], \dots$$
where $i\geq 1$. In particular the nilpotency rank of $U_R(K)$ (as an abstract group) coincides with the nilpotency rank of $V$ (as an algebraic group). Write $e$ for this quantity and note that $e\leq s \leq r$, where $s$ is the nilpotency rank of $U$ (as an algebraic group). The first lemma allows us to ``descend" the lower central series of $V$.

\begin{lem}\label{l: need1}
Fix an integer $i\geq 2$. Suppose that a set $A^*\subset U_R(K)$ satisfies $$A^*/ V^{i-1}(K) = U_R(K)/V^{i-1}(K).$$ Then $(A^*)_k/V^{i}(K) = U_R(K)/V^{i}(K)$ for some $k\ll_{r} 1$.
\end{lem}
\begin{proof}

For $i=2,\dots, e$, define the map
\begin{equation*}
\begin{aligned}
f^i: U_R(K)/ V^{i-1}(K) \times (U_R(K)\cap V^{i-2})(K) / V^{i-1}(K) &\to U_R(K)/ V^{i}(K);\\
(aV^{i-1}(K), bV^{i-1}(K)) &\mapsto [a,b] V^{i}(K).
\end{aligned}
\end{equation*}

Write $F^i$ for $\langle f^i(U_R(K)/V^{i-1}(K), (U_R(K)\cap V^{i-2}(K))/V^{i-1}(K))\rangle$. By the definition of the lower central series, $F^i=V^{i-1}(K)/V^i(K)$. Now observe that $F^i$ is an elementary $p$-group; then $F^i \cong (\Z/p\Z)^{c_i}$ for some positive integer ${c_i}\leq r^2$. We may choose a basis for $F^i$ in the image of $f^i$; thus the basis has form $\{[h_1, k_1], \dots, [h_{c_i}, k_{c_i}]\}$, where $h_l,k_l\in U_R(K)/V^{i-1}(K)$ for $l=1,\dots, {c_i}$.

Now choose $a_l, b_l\in\lieu$ such that $\exp(a_l)V^{i-1}(K)= h_l$ and $\exp(b_l)V^{i-1}(K) = k_l$ for $l=1,\dots, {c_i}$. We proceed similarly to the proof of Lem. \ref{l: alex}. Then
$$f^i(h_l, k_l) = [\exp(a_l),\exp(b_l)]V^i(K) = (1+[a_l,b_l])V^i(K).$$
What is more, for $s,t\in\Z/p\Z$,
$$[\exp(sa), \exp(tb)]V^i(K) = (1+st[a,b])V^i(K).$$
As $s,t$ range over $\Z/p\Z$, the set of these elements forms a subgroup $F^i_l$ of $U_R(K)/V^{i}(K)$ of size $p$. Now observe that
$$ (A^*)_4/ V^i(K)\supseteq f(A^*/V^{i-1}(K), (A^*\cap V^{i-2}(K))/V^{i-1}(K)).$$
We conclude that $(A^*)_4/ V^i(K)$ contains $F^i_l$ for $l=1,\dots, c_i$.

Now, since $\{[h_1,k_1], \dots, [h_{c_i},k_{c_i}]\}$ is a basis for $F^i$, it follows that $F^i=F_1^i\cdots F_{c_i}^i$. Since $c_i\leq r^2$, we conclude that $(A^*)_{4r^2}/V^i(K) \supseteq V^{i-1}(K)/V^i(K)$. Then $(A^*)_{4r^2+1}/V^i(K) =U_R(K)/V^i(K)$ as required.

\end{proof}

\begin{lem}\label{l: need2}
Suppose that a subset $A^*$ of $U_R(K)$ satisfies
$$(S^i_1S^i_2\cdots S^i_{e_i})(K)/ U^{i}(K) \subset A^*/U^i(K)$$ 
for all $i=1,\dots, s$. Then $(A^*)_k/ V^1(K) = U_R(K)/ V^1(K)$ for some $k\ll_{r} 1$.
\end{lem}
\begin{proof}
We prove the result by ``descending" the lower central series of $U$. Observe first that $A^*/V^1(K)U^1(K)$ equals
$$A^*/U^1(K) = (S_1^1\cdots S_{e_1}^1)(K)/U^1(K) = U_R(K)/U^1(K) = U_R(K)/V^1(K)U^1(K).$$
Now fix an integer $i\geq 1$, and assume that $A^*/V^1(K)U^i(K) = U_R(K)/V^1(K)U^i(K)$. Since the nilpotency rank of $U$ is at most $r-1$, it is sufficient to prove that $$(A^*)_k/V^1(K)U^{i+1}(K) = U_R(K)/V^1(K)U^{i+1}(K)$$ for some $k\ll_r 1$.

Observe that 
\begin{equation}\label{pp}
(U_R(K)\cap U^i(K))/V^1(K)U^{i+1}(K)\leq (S^{i+1}_1S^{i+1}_2\cdots S^{i+1}_{e_{i+1}})(K)/V^1(K)U^{i+1}(K).
\end{equation} 
Now $V^1(K)<U_R(K)\leq U(K)$ and $V^1(K)\lhd U(K)$; this means, in particular, that $(U_R(K)\cap U^i(K))V^1(K) = U_R(K)\cap U^i(K)V^1(K)$. It follows that
\begin{equation}\label{ww}
(U_R(K)\cap U^i(K)V^1(K))/V^1(K)U^{i+1}(K)= (U_R(K)\cap U^i(K))/V^1(K)U^{i+1}(K). 
\end{equation}
Since $A^*/V^1(K)U^{i+1}(K)$ contains $(S^{i+1}_1S^{i+1}_2\cdots S^{i+1}_{e_{i+1}})(K)/V^1(K)U^{i+1}(K)$, (\ref{pp}) and (\ref{ww}) imply that $$(A^*)_2/V^1(K)U^{i+1}(K) = U_R(K)/V^1(K)U^{i+1}(K)$$ as required.
\end{proof}

\begin{lem}\label{l: need root groups}
Let $j\geq 1$ be an integer.
Suppose that a set $A^*$ is such that $A^*/U^j(K)$ is a subset of $U_R(K)/U^j(K)$ and
$$(S^i_1S^i_2\cdots S^i_{e_i})(K)/ U^{i}(K) \subset A^*/U^i(K)$$ 
for all $i=1,\dots, j$. Then $(A^*)_k/ U^{j}(K)$ contains $U_R(K)/ U^{j}(K)$ for some $k\ll_{r} 1$.
\end{lem}
\begin{proof}
Observe first that $U_R/ U^{1}$ is equal to $S^1_1S^1_2\cdots S^1_{e_1}$. Thus the statement is true for $j=1$ ($k$ is equal to $1$ in this case).

Now assume the statement is true for $j-1$. Thus there exists $k\ll_{r}1$ such that $(A^*)_k/ U^{j-1}(K)$ contains $U_R(K)/ U^{j-1}(K)$. Note that $(A^*)_k/U^j(K)\subseteq U_R(K)/U^j(K)$. We prove that the statement is true for $j$, and the result follows by induction.

To make matters more transparent, we work inside $G/U^j$; in other words we assume that $U^j$ is trivial. Then, by assumption, the following are true:
\begin{enumerate}
\item $A^* \subseteq U_R(K)$;
\item $(A^*)_k/ U^{j-1}(K) = U_R(K)/U^{j-1}(K)$;
\item $A^* \supseteq (S_1^j S_2^j \cdots S_{e_j}^j)(K)$.
\end{enumerate}

We are required to prove that $(A^*)_{k'}=U_R(K)$ for some $k'\ll_{r} 1$. Observe first that Lem. \ref{l: need2} implies that $(A^*)_{k'}/V^1(K) = U_R(K)/ V^1(K)$ for some $k'\ll_r 1$. 

We apply Lem. \ref{l: need1} with $i=2$. We conclude that $(A^*)_{k''}/V^2(K) = U_R(K)/V^2(K)$ for some $k''\ll_r 1$.

Now we iterate this procedure for $i=3,\dots, e$; since $e\leq r$ we obtain, as required, that $(A^*)_{k'''}=U_R(K)$ for some $k'''\ll_{r} 1$.
\end{proof}

The next lemma allows us to assume that we have elements that ``almost lie on the torus". Recall the definition of $t_R(g)$ given in \S \ref{s: height}.

\begin{lem}\label{l: nice actors}
Suppose that a set $A^*\subseteq G(K)$ contains a set $W^j$ such that $W^j/ U^j(K) = U_R(K)/ U^j(K)$. Then there exists a set $A^\dagger$ in $(A^*)_{r}$ such that $A^\dagger/ U = A^*/U$, and $t_R(g) = 0$ for all $g\in A^\dagger$, and all root subgroups $R$ of height at most $j-1$.
\end{lem}
\begin{proof}
Take $g\in A^*$, and write $g$ in terms of weight subgroup elements:
$$g = x_{R_1}(s_1) x_{R_2}(s_2)\cdots x_{R_k}(s_k)t$$
where $x_{R_i}(s_i)\in R_i$, $t\in T$, and the weights are written in order of increasing height.

Then, by assumption, there exists $h\in A^\dagger$ such that 
$$hU^j(K) = x_{S^1_1}(-s_{1}) \cdots x_{S^1_{e_1}}(-s_{e_1})U^j(K).$$ Now $hg$ has the property that $hgU = gU$, and $t_{R}(hg) = 0$ for all $R\in(\Phi_R^*)^{1}$.

We perform the above procedure $j-1$ times, and we obtain an element $g_0 \in (A^*)_r$ such that $g_0U = gU$ and $t_{R}(hg) = 0$ for all root subgroups $R$ of height at most $j-1$.
\end{proof}

The next step is to show that, under our hypotheses, we can obtain the product of root subgroups of any given height. First a technical lemma similar to Lem. \ref{lem:pqmg}.

\begin{lem}\label{lem:octopus}
Write $G=UT$, and let $E$ be the Cartan subgroup such that $E(\Kbar) = C_{G(\Kbar)}(T(\Kbar))$. Let $g\in G(K)$ be such that $g$ is outside the kernel of every root.
Consider the map 
$$\phi_g: G(K) \to U(K), \, \, h\mapsto \lbrack g,h\rbrack.$$
Then,
\begin{enumerate}
\item $\phi_g(((S_1^i\dotsb S_{e_i}^i)(K))/U^i(K)) = ((S_1^i\dotsb
S_{e_i}^i)(K))/U^i(K)$ for every $i\geq 1$;
\item $\phi_g((U_R U^{j-1})(K)) \subset (U_R U^j)(K)$ for every $j\geq 1$.
\item If we assume that $U^j$ is trivial, and $g,h\in (EU^{j-1})(K)$, then we have that 
\begin{equation*}
\begin{aligned}
& \phi_g(h (U_\Lambda \cap U^i)(K)) \subseteq \phi_g(h) (U_\Lambda\cap U^{i+1})(K)
\end{aligned}
\end{equation*}
for every $i,j\geq 1$.
\end{enumerate}
\end{lem}
\begin{proof}
Note first that $[G,G]=U$, hence the function $\phi_g$ is well-defined.

Consider (a): we are required to prove that the map $\phi_g$ induces a bijection from the group
$(S_1^i\cdots S^i_{e_i})(K)/U^i(K)$ to itself. Suppose that $\phi_g$
were to map two elements $g_1$, $g_2$ to the same element, then $g$
would commute with $g_1 g_2^{-1}$, and this can only happen if 
$g_1 g_2^{-1}$ is the trivial element of
$(S_1^i\cdots S^i_{e_i})(K)/U^i(K)$.

For (b) and (c) note first that, for $h, h'\in G(\Kbar)$,
\begin{equation}\label{e: tentacle}
\begin{aligned}
\phi_g(h h') &= [g,h h'] = g h h' g^{-1} (h')^{-1} h^{-1} = g h g^{-1}
[g,h'] h^{-1} \\ &=
g h g^{-1} h^{-1}\cdot h [g,h'] h^{-1} = \phi_g(h) \cdot h [g,h'] h^{-1}.
\end{aligned}
\end{equation}
Now for (b): take $h\in (U_RU^{j-1})(K)$. We can write $h = h_1 h_2$,
where $h_1\in U_R(\Kbar)$ and $h_2\in U^{j-1}(\Kbar)$. Since
$U_R$ is normal in $G$ we have $\phi_g(h)\in U_R(K)$. Further
$$[g,h_1]\in (S_1^i\cdots S_{e_i}^iU^j)(K).$$
Since $(S_1^i\cdots S_{e_i}^iU^j)(K)$ is normal in $G$, we conclude that 
$$\phi_g(h)\in (U_RS_1^i\cdots S_{e_i}^iU^j)(K) = U_RU^j(K)$$
as required.

Finally (c): take $g,h \in (EU^{j-1})(K)$. Observe that, since $U^{j-1}(\Kbar)$ is central in $G(\Kbar)$, $U_\Lambda(\Kbar)$ is normal in $(EU^{j-1})(\Kbar)$. 


Now consider (\ref{e: tentacle}) with $h'\in (U_\Lambda \cap U^i)(K)$ for some $i\leq j$. We need to show that $h[g,h']h^{-1}\in (U_\Lambda\cap U^{i+1})(K)$. Since $[U^i, U]= U^{i+1}$ and $U_\Lambda(\Kbar)$ is normal in $(EU^{j-1})(\Kbar)$, we conclude that $[g,h']\in (U_\Lambda\cap U^{i+1})(K)$; the result follows.
\end{proof}

\begin{lem}\label{l: got root groups}
Fix $j\geq 1$ an integer. There exists $k\ll_{r} 1$ such that  $A_k$ contains a set $A^*$ such that $A^*/U^j(K)$ is a subset of $U_R(K)/U^j(K)$ and $A^*$ projects surjectively onto
$$(S^i_1S^i_2\cdots S^i_{e_i})(K)/ U^i(K)$$ 
for all $i=1,\dots, j$.
\end{lem}
\begin{proof}
Our hypotheses imply that the lemma is true when $j=1$. We assume that $j>1$ and apply induction, assuming that the statement holds for $j-1$. Thus we assume that there exists $l\ll_{r} 1$ such that $A_l/ U^{j-1}(K)$ contains a set $A^*$ such that $A^*/U^{j-1}(K)$ is a subset of $U_R(K)/U^{j-1}(K)$ and $A^*$ projects surjectively onto
$$(S^i_1S^i_2\cdots S^i_{e_i})(K)/ U^i(K)$$ 
for all $i=1,\dots, j-1$. 

In fact, by working in $G/U^j$ rather than $G$, it is sufficient to assume (as we do from here on) that $U^j$ is trivial. Lem. \ref{l: need root groups} implies that there exists $k\ll_{r} 1$ such that $A_{kl}\cap U_R(K)$ contains a set $X$ such that $X/ U^{j-1}(K) = U_R(K)/ U^{j-1}(K)$.

{\bf Root subgroups of height $j$}. Define the algebraic group $H=EU^{j-1}$, where $E$ is a fixed Cartan subgroup $E =
T\times U_\Lambda$. Observe that, by \cite[\S 7.5]{humphreys3}, $H$ is connected. Now apply Lem. \ref{lem:squid} with $G=G(K)$ and
$H = H(K)$ (We know that $AH/H = G/H$ because we know that
(a)
$A U^{j-1}(K)/U^{j-1}(K)$ contains $U_R(K)/U^{j-1}(K)$ (and so $A H/H
= U_R(K) H/H$)
and (b) $U_R(K) E(K)=G(K)$ (Lem. \ref{l: urul}).)
We obtain $\langle A\rangle
= A \cdot \langle A_3 \cap H(K)\rangle$; thus $\langle A\rangle \cap H(K) =
(A\cap H(K)) \langle A_3 \cap H(K)\rangle \subset \langle A_4 \cap H(K)\rangle$.

Now observe that Lem. \ref{l: nice actors} implies that there exists $k\ll_{r} 1$ such that $A_{k}$ contains a set $A^\dagger$ such that $A^\dagger/ U(K) = A/U(K)$, and $t_R(g) = 0$ for all $R\in(\Phi_R^*)^{<j}$ and $g\in A^\dagger$; in particular, $A^\dagger$ is a subset of $H(K)$. Without loss of generality we assume that $k\geq 4$, and take $A^*=A_k\cap H(K)$.

We write $H$ as a product of unipotent radical and torus, $H=U_HT$, in the usual way. Note that $U_H = U_\Lambda U^{j-1}$ and, in particular, $H$ is of exponential type. Write $H_0 = H/(U_H)^1$, and apply Prop. \ref{p: newinterestgrow} to the set $A^*/ (U_H)^1(K)$.

If (\ref{it:jola}) holds, then Lem. \ref{l: b3} implies a contradiction to (\ref{eq:rusdat}). If (\ref{it:jolb}) holds, then, by Lem.\ \ref{l: b2}, 
$|(A^*)k|\geq C
|A^*|$ for some $k\ll_r 1$ and so, by Lem. \ref{l: b1}, $|A_{k'}|\geq
C|A|$ for osme $k' \ll_r 1$. This is a contradiction to (\ref{eq:rusdat2}).

Thus we conclude that (\ref{it:jolc}) holds: $(A^*)_k/ (U_H)^1(K)$, $k\ll_r 1$, contains a normal subgroup $H$ of $U_H(K)/(U_H)^1(K)$ such that $\langle A^*\rangle/(U_H)^1(K)\rangle/ H$ is abelian.

We are assuming that $U(K)\subset \langle A\rangle$ (by
Lem. \ref{l: containment}). In particular, $(S_1^j \cdots
S_{e_j}^j)(K) \subseteq \langle A\rangle$; hence $(S_1^j \cdots
S_{e_j}^j)(K) \subseteq \langle A_4\cap H\rangle\subseteq \langle A^* \rangle$. 

Observe that $(U_H)^1\leq U_\Lambda$ since $U_\Lambda$ is normal in $U_H$ and $U_H/U_\Lambda$ is abelian. Observe, furthermore, that no element of $(S_1^j \cdots
S_{e_j}^j)(K)$ centralizes $T(K)$. We conclude that $A^*/(U_H)^1(K)\supseteq (S_1^j \cdots
S_{e_j}^j)(K)/ (U_H)^1(K)$, and so $A^*/U_\Lambda(K)\supseteq (S_1^j \cdots
S_{e_j}^j)(K)/U_\Lambda(K)$.

Now (\ref{eq:rusdat}) implies that there exists $g\in A$ lying outside the kernel of every root; Lem. \ref{l: nice actors} implies that we can take $g$ to lie in $A^* \subseteq H(K)$. 
By Lem. \ref{lem:octopus} (a) and (c), this implies that $\phi_g(A_k\cap H(K))$
contains a representative of $h (U_\Lambda(K)\cap U^1(K))$ for every $h\in (S_1^j \cdots
S_{e_j}^j)(K)$ and, iterating, that
$\phi_g^{j}(A_k\cap H(K))$ contains a representative of $h (U_\Lambda \cap U^j)(K)$ for
every $h\in (S_1^j \cdots
S_{e_j}^j)(K)$.
Since $U^j$ is trivial, this means that $\phi_g^{j}(A_k\cap H(K))$
contains $(S_1^j \cdots
S_{e_j}^j)(K)$;
since $\phi_g^{j}(A_k\cap H(K))\subset A_{k'}$, $k'\ll_r 1$, we are done.

{\bf Root subgroups of height $<j$}. We must now examine the groups $(S_1^i\dots, S^i_{e_i})(K)$ for $i=1,\dots, j-1$. We know that for some $k'\ll_{r} 1$, $A_{k'}$ contains a subset $A^*$ such that $A^*/U^{j-1}(K)$ is a subset of $U_R(K)/U^{j-1}(K)$ and $A^*/U^i(K)$ contains 
$$(S_1^i\cdots S^i_{e_i})(K)/U^i(K)$$
for all $i=1, \dots, j-1$. We need to deal with the possibility that $A^*$ is not a subset of $U_R(K)/U^j(K)$. 

Let $g$ be an
element of $A$ such that $g$ is outside the kernel of every root. By
Lem. \ref{lem:octopus}, $\phi_g(A^*) \subset A_{2k'+2}$ satisfies (a)
$\phi_g(A^*)/U^i(K) \supset ((S_1^i\dotsb S_{e_i}^i)(K))/U^i(K)$ for all
$i\leq j-1$,
(b) $\phi_g(A^*) \subset U_R U^j$. We set $k = 2 k' + 2$ and are done.\end{proof}

\begin{cor}\label{c: final}
 Under the hypotheses of this section, there exists $k\ll_{r} 1$ such that $A_k$ contains $U_R(K)$.
\end{cor}
\begin{proof}
 Take $j$ to be the length of the lower central series for $U$; so $U^j=\{1\}$; note that $j< r$. Then we apply Lem. \ref{l: got root groups} using this value of $j$; this implies that there exists $k\ll_{r} 1$ such that $A_k$ contains $(S_1^i\cdots S_{e_i}^i)(K)$ for $i=1,\dots, j$. Now Lem. \ref{l: need root groups} implies that there exists $k'$ such that $A_k'$ contains $U_R(K)$.
\end{proof}

\section{The proof}\label{s: proof}

We are now ready to prove Thm. \ref{t: main}. We abandon all previous hypotheses, except for those given in the statement of the theorem.

\begin{proof}[Proof of Thm. \ref{t: main}]
Take $A\subset \GL_r(K)$ such that $\langle A \rangle$ is solvable. By Prop. \ref{p: bounded borel},
 $\langle A \rangle$ has a subgroup $H$ such that $[\langle A \rangle:H]\ll_r 1$, and $H$ lies in $B(K')\cap \GL_r(K)$ for some Borel subgroup $B/K'$ and
some finite field $K'$. By Prop. \ref{p: bounded degree}, we can assume (as we do) that $\langle A \rangle = H$.

If $p\leq r$ then $|\langle A \rangle|< r^{r^2}$, and so (\ref{i: two})
holds with $S=\langle A \rangle$ and $U_R=O_p(S)$. Assume from here on that $p>r$. 

Let $G$ be as in Lem. \ref{l: containment}. In particular, $G=UT$
is a connected, solvable linear algebraic subgroup of $\GL_r$ defined (and
trigonalizable) over a finite extension of $K$; moreover $G$ is of exponential type in $\GL_r$ and $U(K)\subseteq \langle A \rangle \subseteq G(K)$. Define $\Phi_R$ and $\Phi_R^*$ as usual.

Let $D$ be a positive number (we will fix its value in terms of $C$ in due course). Suppose that $|A_2\cap \ker_G(\alpha(R_j))(K)|\geq \frac1D|A|$ for some $R_j\in\Phi_R^*$. Then redefine $A$ to equal $A_2\cap \ker_G(\alpha(R_j))(K)$, and redefine $G$ to equal $\ker_G(\alpha(R_j))$; note that, by Cor. \ref{c: torus normal}, the dimension of a maximal torus in $G$ has decreased. 

Now test to see whether $U(K)\subseteq \langle A \rangle$; if not, redefine $G$ in line with Lem. \ref{l: containment} so that $U(K)\subseteq \langle A \rangle$. Next test, as before, for a large intersection with a root kernel. Repeat until we have a set $A^*$ and a group $G=UT$ such that $U(K)\subseteq A^*$, and 
$|A^*\cap \ker_G(\alpha(R_j))(K)|\leq \frac1D|A|$ for all $R_j\in\Phi_R^*$.

Since $|U(K)|< p^{r^2}$ and $\dim T < r$, this process must terminate after less than $r^3$ repeats. This means in particular that 
$|A^*|\geq \frac1{D^{r^3}}|A|$.

If $G$ is nilpotent, then we are done; thus we suppose that this is not the case. 
Observe that the assumptions of Section \ref{s: first interesting} are satisfied for $A^*/U^1(K)$ in $(G/U^1)(K)$. We apply Prop. \ref{p: newinterestgrow}. 

If (a) holds, then Lem. \ref{l: b3} implies that $|(A^*)_k\cap \ker_G(\alpha(R))(K)|\geq \frac1D|A|$, for some $R\in\Phi_R^*$ and some $k\ll_r 1$; this is a contradiction.

If (b) holds, then $|(A^*/U^1(K))_k|\geq D|A^*/U^1(K)|$. An application of Lem. \ref{l: b2} implies that
$$|(A*)_{4k}|\geq D |A^*| \geq D|AA^{-1}\cap G(K)|.$$
Then Lem. \ref{l: b1} implies that $|A_{4k+1}|\geq D|A|$, and finally Lem. \ref{l: tripling} implies that 
$$|A_3|\geq D^\delta |A|$$
for some $\delta\ll_r 1$. Now fix $D=C^{\frac1\delta}$ and Thm. \ref{t: main} is proved.

Finally we assume that (c) holds. Then $(A^*)_k/ U^{1}(K)$ contains the non-trivial subgroup $U_R(K)/U^1(K)$ for some $k\ll_{r} 1$, and the hypotheses of Section \ref{s: hypotheses} are all fulfilled for the set $(A^*)_k$ lying in $G(K)$.

Cor. \ref{c: final} implies that there exists $k'\ll_{r} 1$ such that $(A^*)_{kk'}$ contains $U_R(K)$. Now $U_R(K)$ is normal in $G(K)$, and Lem. \ref{l: urul} implies that $G(K)/ U_R(K)$ is nilpotent. Set $S=\langle A \rangle \cap G(K)$; we know that $|A^*|\geq \frac1{D^{r^3}}|A|$ and so $|A_k\cap G(K)| \geq C^{-\frac{r^3}{\delta}}|A|$ with $k\ll_r 1$.

We are almost done: we know that $U_R(K)$ is normal in $S$; if $U_R(K)$ is normal in $\langle A \rangle$, then set $U_R:=U_R(K)$ and we are finished. Suppose instead that $U_R(K)$ is not normal in $\langle A \rangle$. 

Prop. \ref{p: bounded borel} implies that $\langle A \rangle$ contains a subgroup $H$ such that $[\langle A \rangle:H]\ll_r 1$, and $H$ lies in $B(K')\cap G(K)$ for some Borel subgroup $B/K'$ and some finite field $K'$. 

If $p$ is bounded above by a function of $r$ then the same is true for the order of a Borel subgroup of $\GL_r(K)$. Now Prop. \ref{p: bounded borel} implies that the same is true for the order of any abstract solvable subgroup in $\GL_r(K)$. This in turn implies that (\ref{i: two}) holds with $S=\langle A \rangle$ and $U_R=O_p(S)$). 

We assume, therefore, that $p$ is not bounded above by a function of $r$; in particular we take $p$ to be greater than $[\langle A \rangle:H]$. This implies that a Sylow $p$-subgroup of $H$ is a Sylow $p$-subgroup of $\langle A \rangle$. Since $H$ lies in a Borel subgroup of $\GL_r(K)$, a Sylow $p$-subgroup of $H$ is normal in $H$; it is equal to $O_p(H)$. All Sylow $p$-subgroups of $\langle A \rangle$ lie in $H$, hence they all coincide with $O_p(H)$; we conclude that $O_p(H)$ is normal in $\langle A \rangle$ and is equal to $O_p(\langle A \rangle)$. 

For $a\in A$ and $H\leq S$ we write $H^a$ to mean the conjugate $aHa^{-1}$. We give an algorithm to produce the required group $U_R$. Start by setting $U_R:=U_R(K)$ and fix $a\in A$ so that $U_R\neq U_R^a$. Since $O_p(H)$ is normal in $\langle A \rangle$ and $U_R\leq O_p(H)$ we have that $U_R^a\leq O_p(H)$ for all $a\in A$. In particular $|U_R\cdot U_R^a|\geq p|U_R|$. Furthermore, since $U_R\unlhd S$, we have that $U_R^a \unlhd S^a=S$ and so $U_R\cdot U_R^a \unlhd S$. Finally, observe that $U_R\cdot U_R^a \subseteq A_{2k+2}$.



If $U_R\cdot U_R^a$ is normal in $\langle A \rangle$ then we are done: we redefine $U_R$ to be $U_R\cdot U_R^a$ and $k$ to be $2k+2$, and (\ref{i: two}) holds with $U_R\unlhd \langle A \rangle$. If $U_R\cdot U_R^a$ is non-normal in $\langle A \rangle$ then we may repeat the above argument - choosing $a'$ such that $(U_R\cdot U_R^a)^{a'}\neq U_R\cdot U_R^a$ to yield a still larger group $(U_R\cdot U_R^a)^{a'}\cdot(U_R\cdot U_R^a)\subset A_{2(2k+2)+2}$.  Now a chain of unipotent subgroups of $\GL_r(K)$,  $U_1 > U_2 >\cdots$, has length less than $r^2$, and so we can repeat the above process less than $r^2$ times before we yield a subgroup $U_R'$ which lies in $A_{k'}$ for some $k'\ll_r 1$, which is normal in $\langle A \rangle$  and which, along with the subgroup $S$, satisfies all the conditions of Thm.~\ref{t: main}.
\end{proof}

\section{Theorem~\ref{t: main2}}\label{s: extension}

In this section we prove Thm.~\ref{t: main2}, which is an extension of Thm.~\ref{t: main} to the situation where $\langle A\rangle$ is not necessarily solvable. Our proof uses Thm.~\ref{t: main} as well as a result of Pyber and Szab{\'o}; Thm.~\ref{t: main2} should  be considered joint work with them. We begin with the key result of Pyber and Szab{\'o}.

\begin{thm}\label{t: ps}\cite[Cor. 103]{ps2}
Let $K=\Z/p\Z$ and let $A$ be a subset of $\GL_r(K)$ such that $A=A^{-1}$. Then, for every $C\geq 1$, either
\begin{enumerate}
\item\label{i: ps1} $|A\cdot A\cdot A| \geq C|A|$, or else
\item\label{i: ps2} there are two subgroups $P\leq H \leq \GL_r(K)$, both normal in $\langle A \rangle$, such that
\begin{itemize}
\item $P$ is perfect, and $H/P$ is solvable;
\item a coset of $P$ is contained in $A\cdot A\cdot A$; and 
\item $A$ is covered by $C^{O_r(1)}$ cosets of $H$.
\end{itemize}
\end{enumerate}
\end{thm}

We can drop the condition that $A=A^{-1}$ provided we replace occurrences of $A\cdot A \cdot A$ in the statement with $A_3$. Thm. \ref{t: ps} effectively reduces the study of growth in $\GL_r(K)$ to the study of growth in solvable sections of $\GL_r(K)$. 

Next we reproduce \cite[Prop. 105]{ps2} (including a proof for completeness):

\begin{prop}\label{p: reduction}
Let $H$ be a finite group and $P$ a normal subgroup with $H/P$ solvable. If $F$ is a minimal subgroup such that $PF=H$ then $F$ is solvable.
\end{prop}
\begin{proof}
Let $M$ be a maximal subgroup of $F$. If $M$ does not contain $F\cap P$ then $(F\cap P)M=F$ which implies $PM=PF=H$, a contradiction. Hence all maximal subgroups of $F$, and therefore $\Phi(F)$, the Frattini subgroup of $F$, contain $F\cap P$. But $\Phi(F)$ is nilpotent \cite[5.2.15]{robinson} and so $P\cap F$ is nilpotent. Now $F/F\cap P \cong PF/P = H/P$ is solvable; we conclude that $F$ is solvable.
\end{proof}

We need some simple technical lemmas; the first is a strengthening of Lem. \ref{lem:duffy} for normal subgroups.

\begin{lem}\label{lem:duffy2}
Let $G$ be a group and $H$ a normal subgroup thereof. Let $A\subset G$ be a 
non-empty finite set. Let $l$ be the number of cosets of $H$ intersecting $A$, and set $B=AA^{-1}\cap H$.
There are $l$ elements $a_1, \dots, a_l\in A$ such that $A$ is contained in
$a_1B\cup \cdots \cup a_lB$.
\end{lem}
\begin{proof} 
Let $c\in G$ so that $cH\cap A$ is non-empty. Fix $a_1=ch\in cH\cap A$; for any element $ch'\in CH\cap A$ we have
$$ch' = (ch')(ch)^{-1}(ch) = (ch'h^{-1}c^{-1})(ch) \in B(ch) = (ch)B.$$
We can repeat this process for each coset such that $cH\cap A$ is non-empty; since there are only $l$ of these, the result follows.
\end{proof}

\begin{lem}\label{l: l}
 Let $R, R'$ be subgroups of a group $G$. Let $A,B$ be subsets of $G$. Then
$$|AB|\geq \frac{|A\cap R|\cdot |B\cap R'|}{|AA^{-1}\cap R \cap R'|}.$$
\end{lem}
\begin{proof}
 It is obvious that $|AB|\geq |(A\cap R)\cdot(B\cap R')|$. Now if distinct pairs $(x,y), (x', y') \in (A\cap R)\times (B\cap R')$ have the same image under the multiplication map $(x,y)\mapsto xy$, then $x^{-1}x' = y(y')^{-1}$, and so $x^{-1}x'$ lies in both $R$ and $R'$.
\end{proof}

\begin{lem}\label{l: m}
 Let $R$ be a subgroup of a group $G$. Let $A$ be a subset of $G$, and $a$ an element of $A$. Then
$$|A_4|\geq \frac{|A\cap R|^2}{|AA^{-1} \cap R \cap aRa^{-1}|}.$$
\end{lem}
\begin{proof}
 First of all, notice that 
$$|AAA^{-1}\cap aRa^{-1}\geq |aAa^{-1}\cap aRa^{-1}|=|A\cap R|.$$
Now apply Lem. \ref{l: l} with $R'=aRa^{-1}$ and $B=AAA^{-1}$.
\end{proof}

For the final part of the proof of Thm. \ref{t: main2} we will need the concept of the {\it degree} of an algebraic variety. Rather than give a full treatment of this concept we refer the reader to \cite[\S 2.5.2]{helfgott3} where, for an affine algebraic variety $V$, the degree $\vdeg(V)$ is defined as a vector
$$(d_0 , d_1 , \dots , d_k , 0, 0, 0, \dots ),$$
where $k = \dim(V)$ and $d_j$ is the degree of the union of the irreducible components of $V$ of
dimension $j$. If $V$ is pure-dimensional, then $\vdeg(V)$ has only one non-zero entry which we write $\deg(V)$. 

We will need the version of Bezout's theorem given as \cite[Lem. 2.4]{helfgott3} and proved in \cite[p. 251]{danilov}:

\begin{lem}\label{l: bezout}
Let $X_1 , X_2 , \dots , X_k$ be pure-dimensional varieties in $\mathbb{P}_n$; let $Z_1 , Z_2 , \cdots , Z_l$ be the irreducible components of the intersection 
$X_1 \cap X_2 \cap \cdots \cap X_k$. Then

$$\sum_{j=1}^l\deg(Z_j ) \leq \prod_{i=1}^k\deg(X_i ).$$
\end{lem}

In order to state some consequences of this result we need some notation. Write $\overrightarrow{d}$ to mean a vector of integers $(d_1,\dots, d_k, 0,0,\dots)$ for which all entries are zero after some finite index $k$. We say that the vector $\overrightarrow{d}$ is {\it bounded above in terms of a variable $r$} if $k\ll_r 1, d_1\ll_r 1, \dots,$ and $d_k \ll_r 1$. Similarly, a vector $(d_1,\dots, d_k, 0,0,\dots)$ {\it is bounded above in terms of vectors $\overrightarrow{e_1}, \dots \overrightarrow{e_n}$} if the numbers $k,d_1,d_2,\dots, d_k$ are bounded above by functions depending only on the number of non-zero entries in $\overrightarrow{e_1},  \dots , \overrightarrow{e_n}$, and on the value of those entries.

It is easy to see that Bezout's theorem implies that, for any varieties $V_1 , V_2 , . . . , V_k$ (pure-dimensional or otherwise), the
degree $\vdeg(W)=(d_1,\dots, d_k, 0,0,\dots)$ of the intersection $W = V_1 \cap V_2 \cap \cdots \cap V_k$ is bounded above in terms of $\vdeg(V_1), \vdeg(V_2 ), \dots , \vdeg(V_k )$ alone.

We will apply Bezout's theorem via the following two results; the proof of the first is based on the proof of \cite[Prop. 4.1]{helfgott3}. We need one more definition: for an algebraic variety $X$ of dimension $d$ define the {\it dimension vector} of $X$ to be the vector $(s_0, s_1, \dots, s_d, 0, 0, \dots)$ where $s_i$ is the number of components of $X$ of dimension $i$.

\begin{lem}\label{l: dim vector}
 Let $X$ and $Y$ be varieties in $\mathbb{P}_n$ such that $X\subsetneq Y$. Write \[(s_0, s_1, \dots, s_k, 0, 0, \dots)\] (resp. $(t_0, t_1, \dots, t_l, 0, 0, \dots)$) for the dimension vector of $X$ (resp. $Y$). There exists a non-negative integer $m$ such that if $n>m$ then $t_n=s_n$, and $t_m < s_m$. 
\end{lem}
\begin{proof}
For $i\in\mathbb{N}$ write $X_i$ (resp. $Y_i$) for the union of components of $X$ (resp. $Y$) of dimension $i$. Let $m$ be the minimum integer such that $X_n=Y_n$ for all $n>m$; since $X\neq Y$ are 
distinct we know that $m\geq 0$. Clearly $t_n=s_n$ for $n>m$. Clearly $Y$ does not contain all of $X_m$, thus the number of components of $Y_m$ is $t_m< s_m$.  
\end{proof}

\begin{cor}\label{c: bounded iteration}
Let $\{X_i: i\in \mathbb{N}\}$ be a set of distinct varieties in $\mathbb{P}_n$ whose degree vectors are bounded above uniformly in terms of some variable $r$. There exists an integer $N\ll_r 1$ such that if
\begin{equation}\label{peach}
X_0 \supsetneq X_0\cap X_1 \supsetneq X_0\cap X_1 \cap X_2 \supsetneq \cdots \supsetneq X_0\cap X_1 \cap X_2 \cap \cdots \cap X_n,
\end{equation}
then $n<N$.
\end{cor}
\begin{proof}
Suppose that (\ref{peach}) holds for some $n$. Since the degree vector of $X$ is bounded above in terms of $r$, so too is the dimension vector of $X$. Now apply Lem. \ref{l: dim vector} repeatedly, first with $X=X_0$ and $Y=X_0\cap X_1$, then with $X=X_0\cap X_1$ and $Y=X_0\cap X_1\cap X_2$, etc. Lem. \ref{l: bezout} (and the comments after it) implies that, after $m\ll_r 1$ iterations, either $X=Y$ (and the result follows) or the dimension vector of $Y$ has form $(t_0, 0,\dots, 0)$; what is more $t_0\ll_r 1$. In this case the variety $X$ consists of $t_0$ points. We can apply Lem. \ref{l: dim vector} at most a further $t_0$ times; either $X=Y$ holds before we complete these iterations (and the result follows), or else $X_0\cap X_1 \cap \cdots \cap X_N$ is the empty variety, and the result follows.
\end{proof}

In order to apply Bezout's theorem we will need information about the degree of some varieties that we have already encountered.

\begin{lem}\label{l: nice kitty}
 Let $A\subset B(K)$, where $K=\Z/p\Z$ and $B$ is a Borel subgroup of $\GL_r$. Let $G$ be the connected, solvable $K'$-group $G=UT$ defined in Lem. \ref{l: containment}. Let $\Phi_R^*$ be a set of roots for $G$. Then
\begin{itemize}
 \item $G$ is an affine algebraic variety of degree bounded above in terms of $r$;
\item Let $\eta_1, \dots, \eta_m\subset \Phi_R^*$; then $G_I= \ker_G(\eta_1)\cap \cdots \cap \ker_G(\eta_m)$ is an affine algebraic variety of degree bounded above in terms of $r$.
\end{itemize}
\end{lem}

\begin{proof}
The group $G=UT$ where $U$ and $T$ are varieties lying in affine subspaces $\mathbb{A}_1$ and $\mathbb{A}_2$ which intersect only in $\{e\}$; thus, to bound the degree of $G$, it is sufficient to bound the degree of $U$ and $T$.

The group $U$ is constructed in Lem. \ref{l: container}; it is defined by equations $f_i(\log X)$ for some linear functions $f_i$; thus, in particular, $U$ has degree bounded above in terms of $r$.

Write $B=U_rT_r$ for the decomposition into torus and unipotent radical; then the group $T=N_{T_r}(U)$; this group is considered in Lem. \ref{l: torus normal}. The group $T_r$ is conjugate to the set of invertible diagonal matrices; this set is defined by equations of degree at most $r+1$. Then the proof of Lem. \ref{l: torus normal} implies that to define $T$ we require only the equations defining $T_r$ as well as some linear equations; we conclude that $T$, and hence $G$, has bounded degree.

Now the proof of Cor. \ref{c: torus normal} implies that the group $G_I$ is defined as a subset of $G$ by linear equations; hence it too has bounded degree.
\end{proof}

We are ready to prove Thm. \ref{t: main2}.

\begin{proof}
Take $A$ as prescribed, and apply Thm. \ref{t: ps} to $A\cup A^{-1}\cup \{1\}$. If (\ref{i: ps1}) holds, then $|A_3|\geq C|A|$ and we are done. Suppose instead that (\ref{i: ps2}) holds; then we have two subgroups $P\leq H \leq \GL_r(K)$ with the given properties. Note that the group $P$ is a subset of $A_3A_3^{-1}$.

Next apply Prop. \ref{p: reduction} to the two subgroups $P$ and $H$; we obtain a solvable subgroup $F\leq \GL_r(K)$ such that $PF=H$. Define $A'=A_3A_3^{-1}\cap H$ and consider the natural projection map
$$\pi: H\to H/P = PF/P \cong F/F\cap P.$$
Now $\pi(A')$ can be thought of as a subset of $F/F\cap P$; write $D$ for the full pre-image of $\pi(A')$ in $F$.

 We apply Thm. \ref{t: main} to $D$ with constant $C^{47}$. If (\ref{i: one}) holds, then $|D_3|\geq C^{47}|D|$. Since $D$ is the full pre-image of $\pi(A')$ this implies that $|(\pi(A'))_3|\geq C^{47}|\pi(A')|$. Now Lem. \ref{l: b2} implies that $|(A')_8|\geq C^{47}|A'|$; since $(A')_8 \subseteq A_{48}\cap H$ and $A'\supseteq
A^{-1} A \cap H$,  Lem. \ref{l: b1} implies that $|A_{49}|\geq C^{47}|A|$; finally the Tripling Lemma yields that $|A_3|\geq C|A|$ and we are done.

Suppose that (\ref{i: one}) of Thm. \ref{t: main} does not hold with respect to $D$. Then (\ref{i: two}) holds and we obtain two groups, $S\leq F$ and $U_R\leq F$, with the given properties. In particular, since $K=\Z/p\Z$ we know that both $S$ and $U_R$ are normal in $\langle D \rangle$. 

Let $\phi: F\to F/F\cap P$ be the natural projection map; observe that $\phi(D)=\pi(A')$. It is easy to check that the conclusions of Thm. \ref{t: main} apply to $\pi(A')$ as a subgroup of $F/F\cap P$; that is to say the subgroups $\phi(S)$ and $\phi(U_R)$ are normal subgroups of $\langle \pi(A') \rangle$ such that $\phi(S)/\phi(U_R)$ is nilpotent, $(\pi(A'))_{k'}$ contains $\phi(U_R)$ and $\pi(A')$ is contained in $C^{O_r(1)}$ cosets of $\phi(S)$. Here $k'$ depends only on $r$.

Now we take the preimage, $\pi^{-1}$, of all of these objects in $H$. We obtain groups $S'=\pi^{-1}(\phi(S))$ and $U_R'= \pi^{-1}(\phi(U_R))$ such that $S'/U_R'$ is nilpotent and $A'$ lies in $C^{O_r(1)}$ cosets of $S'$. What is more, since $A'$ contains $P$ and $(\pi(A'))_{k'}$ contains $\phi(U_R)$, we conclude that $U_R'$ lies in $(A')_{k'+1}$ . Recall that $A$ lies in $C^{O_r(1)}$ cosets of $H$ by Thm. \ref{t: ps}; hence, by Lem. \ref{lem:duffy2}, $A$ lies in $C^{O_r(1)}$ translates of $A'$; together these facts imply that $A$ lies in $C^{O_r(1)}$ cosets of $S'$.

There is one problem remaining: the groups $U_R'$ and $S'$ need not be normal in $\langle A \rangle$.  Observe that $\langle A \rangle$ acts as an automorphism group of the group $H/P$ (since $H$ and $P$ are both normal in $\langle A \rangle$). Recall that $H/P\cong F/F\cap P$ where $F$ is a solvable subgroup of $\GL_r(K)$.

{\bf 1. The group $H_1$ can be chosen to be normal.}
By Prop. \ref{p: bounded borel} we know that $F$ intersects $B(K)$ for some Borel subgroup $B$ such that $F_0=F\cap B(K)$ is normal in $F$ and $|F:F_0|\ll_r 1$. Note that the group $P_0=O_p(F_0)$ is a $p$-group normal in $F_0$, and $F_0/P_0$ is abelian of order coprime to $p$. We may assume that $p$ is larger than any function of $r$ (since, otherwise, Thm. \ref{t: main2} follows trivially - (\ref{i: two2}) holds with $H_1=H_2=\langle A \rangle$). Then we can take $p>|F:F_0|$ and so $P_0$ is normal in $F$; indeed we have that $(|F/P_0|, p)=1$ and so $P_0$ is a normal Sylow $p$-subgroup of $F$, hence is characteristic in $F$.

Since the group $U_R$ specified in Thm. \ref{t: main} is unipotent, it is a $p$-group, and we know that $U_R$ is a subgroup of $P_0$. Since $P_0$ is characteristic in $F$, the action of $\langle A \rangle$ on $H/P\cong F/F\cap P$ induces an action on $P_0/(F\cap P)$. Let $aU_Ra^{-1}/ (F\cap P)$ be a conjugate of $U_R/ (F\cap P)$ by an element of $A$ that is not equal to $U_R/ (F\cap P)$. Then $U_RaU_Ra^{-1}$ is a subgroup of $P_0$ that is strictly larger than $U_R$. Since $P_0$ has subgroup chains $P_0>P_1>\cdots$ of length at most $r^2$, we can only repeat this process at most $r^2$ times until we obtain a subgroup $H_1'$ of $P_0/(F\cap P)$ that is normalized by $\langle A \rangle$ (in the induced action on $P_0/(F\cap P)$). The preimage in $S$ of $H_1'$ is a normal subgroup, $H_1$, of $\langle A \rangle$  lying in $A_{k''}$ for some $k''\ll_r 1$. Since it is strictly greater than $U_R'$ we know that $S'/H_1$ is nilpotent.

{\bf 2. The group $H_2$ can be chosen to be normal.} We begin with a claim: {\it The group $S$ in $F$ is equal to $\langle D_B \rangle \cap G_0(K)$ where $G_0$ is an algebraic group of degree bounded above in terms of $r$, $D_B$ is some subset of $D_l\cap B(K)$ for some $l\ll_r 1$, and $G_0(K)/ U_R$ is nilpotent.} 

To prove the claim, we must recall how the group $S$ was constructed in the proof of Thm. \ref{t: main}. The first reduction comes via Prop. \ref{p: bounded degree} in which $S$ is constructed as the intersection of $\ll_r 1$ conjugates of $S_H$, a subgroup of $\langle D_B\rangle$ for $D_B$ some subset of $D_l\cap B(K)$. Lem. \ref{l: bezout} implies that it is sufficient to prove that $S_H = \langle D_B\rangle \cap G_1$ where $G_1$ is a linear algebraic group of degree bounded above in terms of $r$. 

Let $G$ be the linear algebraic group from Lem. \ref{l: containment} with $A=D_B$.
The proof of Thm. \ref{t: main} given in \S \ref{s: proof} defines $S$ to be $\langle D_B\rangle \cap G_1(K)$ where $G_1$ is the intersection of a number of root kernels in $G$; now Lem. \ref{l: nice kitty} implies that $G_1$ has degree bounded above in terms of $r$.

Finally observe that the group $U_R$ is constructed with respect to $G_1$ so that $G_1(K)/U_R$ is nilpotent. Since $G_0\leq G_1$ we conclude that $G_0(K)/U_R$ is nilpotent and the claim is proved.

Now suppose that $G_0$ is not normalized by the action of $\langle A \rangle$ on $H/P$. Thm. \ref{t: main} implies that there exists $\delta\ll_r 1$ and $k\ll_r 1$ such that $|D_k\cap S|\geq C^{-\delta}|D|$. 

Suppose that $|D_kD_k^{-1}\cap S \cap aSa^{-1}|\leq C^{-2\delta -\frac{4k-2}{15}}|D|$ for some $a\in \langle A \rangle$. We apply Lem. \ref{l: m} with $R=S$ and $A=D_k$ to obtain that
$$|D_{4k}|\geq \frac{|D_k\cap S|^2}{|D_kD_k^{-1}\cap S \cap aSa^{-1}|} \geq \frac{C^{-2\delta}|D_k|^2}{C^{-2\delta-\frac{4k-2}{15}}|D_k|} = C^{\frac{4k-2}{15}}|A|.$$
An application of Lem. \ref{l: tripling} implies that $|D_3|\geq \sqrt[15]{C}|D|$ and, just as before, this implies that $|A_3|\geq C|A|$ and so (\ref{i: one2}) holds and we are done.

Suppose, instead, that $|D_kD_k^{-1}\cap S \cap aSa^{-1}| \geq C^{-2\delta -\frac{4k-2}{15}}|D|$ for all $a\in \langle A \rangle$. Then $|D_{2k}\cap (S\cap aSa^{-1})|\geq C^{-2\delta -\frac{4k-2}{15}}|D|$, and (\ref{i: two}) of Thm. \ref{t: main} holds with $S$ replaced by $S\cap aSa^{-1}$, $k$ replaced by $2k$, and $\delta$ replaced by $2\delta +\frac{4k-2}{15}$.

We iterate this procedure, choosing elements $a_1, a_2, \dots$ so that
\begin{equation}\label{e: pear} 
G_0 > G_0 \cap a_1G_0a_1^{-1} > G_0 \cap a_1G_0a_1^{-1}\cap a_2 G_0 a_2^{-1} > \cdots.
\end{equation}
Note that all containments here are strict. If, at any point, we obtain growth, i.e., $|A_3|\geq C|A|$, then we are done as (\ref{i: one2}) of Thm. \ref{t: main2} holds. Suppose that this does not happen. Then we apply Cor. \ref{c: bounded iteration} with $X_0= G_0, X_1=a_1G_0a_1^{-1}, X_2=a_2 G_0 a_2^{-1}$ and so on. We conclude that there are at most $m\ll_r 1$ elements $a_1, \dots a_m$ which satisfy (\ref{e: pear}). Thus the intersection $G_0 \cap a_0G_0a_0^{-1}\cap a_1G_0a_1^{-1}\cap \cdots \cap a_mG_0a_m^{-1}$ is normalized by the action of $\langle A\rangle$. We call this intersection $H$ and note that, in particular, $D$ lies in $C^{O_r(1)}$ cosets of $H(K)$.

Now write $D_1 = \langle D^a \, \mid \, a\in\langle A \rangle\rangle$, and set $H_2'=D_1\cap H(K)$. This is normalized by the action of $\langle A \rangle$ on $H/P$, and hence $H_2=\pi^{-1}(H_2')$ is a normal subgroup of $\langle A \rangle$. Since $G_0(K)/ U_R$ is nilpotent we know that $H_2/H_1$ is nilpotent. Finally, since $D$ lies in $C^{O_r(1)}$ cosets of $G_0(K)$, we conclude that $A'$ lies in $C^{O_r(1)}$ cosets of $H_2$, and Lem. \ref{lem:duffy2} implies that $A$ lies in $C^{O_r(1)}$ cosets of $H_2$.
\end{proof}

\section{Theorem~\ref{t: main3}}\label{s: extension 2}

In this section we prove Thm.~\ref{t: main3}. Before we do this, we must explain the three new pieces of terminology that were used in the statement of Thm.~\ref{t: main3}; the first two are due to Tao \cite{taofrei, taononcomm}; the third was also first defined by Tao \cite{taofrei}, however we prefer to work with the slightly different definition of \cite{tointon}, which is in line with that in \cite{bgt11}. In what follows we set $G$ to be a group and $C>1$, a real number.

We define a subset $A\subset G$ to be a {\it $C$-approximate group} if $A=A^{-1}$ and there exists $X\subseteq G$ such that $X=X^{-1}$, $|X|\leq C$ and $AA\subseteq XA$.

For two subsets $A,B\subset G$, we say that $A$ is {\it $C$-controlled} by $B$ if $|B|\leq C|A|$ and there exists $X\subseteq G$ such that $|X|\leq C$ and $A\subseteq XB\cap BX$.

Finally we need the notion of a {\it coset nilprogression}, which we define in two stages as follows. 

Let $x_1,\dots, x_r$ be elements that generate a nilpotent group of nilpotency class $s$ and let $L=(L_1,\dots, L_r)$ be a vector of positive integers. Then the set of all products in the $x_i$ and their inverses, in which each $x_i$ and its inverse appear at most $L_i$ times between them, is called a \emph{nilprogression of rank $r$ and step $s$}.

Now a \emph{coset nilprogression of rank $r$ and step $s$} is a subset of $G$ of the form $\pi^{-1}(Q)$, where $G_0$ is a subgroup of $G$, $H$ is a finite normal subgroup of $G_0$, $\pi: G_o\to G_o/H$ is the quotient map, and $Q$ is a nilprogression of rank $r$ and step $s$ in $G_0/H$.

In what follows we will denote a coset nilprogression of this form by $HP$, in order to emphasise that it is a collection of cosets of the subgroup $H$. The set $P$ appearing in this notation is not, in general, uniquely defined, a fact that will not affect anything that follows.


We need to connect these new notions to growth, and the next two results do just that. The first is due to Tao \cite{taononcomm}; the formulation given here can be found as part of \cite[Prop. 3.1]{BG1}.

\begin{lem}\label{l: connect}
Let $A$ be a set in a group $G$ and $C>1$, a real number.
\begin{enumerate}
\item If $|AAA|\leq C|A|$, then the set
$$B:=\{a_1a_2a_3 \mid a_1,a_2,a_2\in A\cup A^{-1}\}$$
is a $C^{O(1)}$-approximate group and $A$ is $C^{O(1)}$-controlled by $B$.
\item If $1\in A$ and $A$ is a $C$-approximate group, then $|A_3|\leq C^2|A|$.
\item If $A$ is a $C$-approximate group, then $A^n$ is $C^{n+1}$-controlled by $A$.
\end{enumerate}
\end{lem}

We now state the key result of Tointon \cite[Thm.~1.4]{tointon}

\begin{thm}\label{t: tointon}
 Let $G$ be a nilpotent group of nilpotency class $s$, and let $A\subset G$ be a $C$-approximate group. Then there exists a coset nilprogression $HP$ of rank $C^{O_s(1)}$ such that
 $$A \subseteq HP \subseteq A^{C^{O_s(1)}}.$$
\end{thm}

\begin{cor}\label{c: tointon}
  Let $G$ be a nilpotent group of nilpotency class $s$, and let $A\subset G$ be a $C$-approximate group. Then $A$ is $\exp(C^{O_s(1)})$-controlled by a coset nilprogression of rank $C^{O_s(1)}$ contained in $A^{C^{O_s(1)}}$.
\end{cor}
\begin{proof}
 Lem.~\ref{l: connect} implies that $A^{C^{O_s(1)}}$ is
 $C^{C^{O_s(1)}}$-controlled by $A$, i.e.,  $A^{C^{O_s(1)}}$ is
 $\exp(C^{O_s(1)})$-controlled by $A$. Now Thm.~\ref{t: tointon} tells us
 that $A^{C^{O_s(1)}}$ contains a coset nilprogression $HP$ of rank
 $C^{O_s(1)}$ containing $A$. It follows that $A^{C^{O_s(1)}}$ is 
 $\exp(C^{O_s(1)})$-controlled by $HP$, and so $A$ is
 $\exp(C^{O_s(1)})$-controlled by $HP$ as well.
\end{proof}

We need one final lemma due to Tao \cite[Lem. 3.6]{taononcomm}; it is the non-abelian analogue of Ruzsa's covering lemma.

\begin{lem}\label{l: ruzsa cover}
 Let $A,B$ be finite subsets of a group $G$ and $C>1$. If $|B\cdot A| \leq C |B|$ (resp. $|A\cdot B| \leq C |B|$), then there exists a finite set $Y\subseteq A$ such that $|Y|\leq C$ and $A\subseteq B^{-1}BY$ (resp. $A\subseteq YBB^{-1}$).
\end{lem}

We can now prove the main result of this section.

\begin{proof}[Proof of Thm.~\ref{t: main3}]
We assume, without loss of generality, that $1\in A$; since $A$ is symmetric this implies that $A_3=AAA$. Now Lem.~\ref{l: connect} implies that $|A_3|\leq C^2|A|$. 
 
It will be convenient to assume that $C^2>2$. If this were not the case, then Lem.~\ref{l: olson} implies that $A_3=\langle A \rangle$ and the result holds with the coset nilprogression taken to be $\langle A \rangle$.
 
Now we apply Thm.~\ref{t: main2} with constant $C^2$ and conclude that (b) holds - let $H_1$ and $H_2$ be the given subgroups, $k$ the given positive integer such that $A_k \supseteq H_1$.
 
Let $A' = A_{2k}\cap H_2$. Then Lem.~\ref{l: b1} and
Lem.~\ref{l: tripling}(b)
 imply that
$$\frac{|A'_3|}{|A'|} = \frac{|(A_{2k}\cap H_2)_3|}{|A_{2k}\cap H_2|} \leq \frac{|A_{6k}\cap H_2|}{|A_{2k}\cap H_2|} \leq \frac{|A_{6k+1}|}{|A|}\leq (C^2)^{6k-1}.$$

We apply Prop.~\ref{p: reduction} to obtain a solvable subgroup $F< \GL_r(K)$ such that $H_1F=H_2$. Consider the natural projection
$$\pi: H_2 \to H_2/H_1 = H_1F/H_1 \cong F/(F\cap H_1).$$
Then Lem.~\ref{l: b2} implies that
\begin{equation}\label{e: d grows}
\frac{|\pi(A')_3|}{|\pi(A')|} = \frac{|\pi(A'_3)|}{|\pi(A')|} \leq \frac{|A'_8|}{|A'|}= C^{O_r(1)}.
\end{equation}

Prop.~\ref{p: bounded borel} implies that $F$ has a normal subgroup $Q_F$ such that $|F:Q_F|\ll_r 1$ and $Q_F$ is a subgroup of $B(K)$, where $B$ is a Borel subgroup of $\GL_r$ defined and trigonalizable over $K'$, a field extension of $K$ of degree at most $r$. Since $B(K)$ has abelian Sylow $t$-subgroups for $t\neq p$, and a unique Sylow $p$-subgroup of nilpotency class at most $r$, any nilpotent section of $Q_F$ has nilpotency class at most $r$.

Write $D$ for the set $\pi(A')$ and write $Q$ for the image in $H_2/H_1$ of $Q_F(F\cap H_1)/(F\cap H_1)$ under the isomorphism $F/(F\cap H_1)\to H_2/H_1$. In particular, since $Q_F(F\cap H_1)/(F\cap H_1) \cong Q_F/(F\cap H_1\cap Q_F)$, $Q$ is nilpotent of class at most $r$. 

Prop.~\ref{prop:usef} implies that there
are subsets $D_{Q,1}, D_{Q,2}\subset D_k \cap Q$, $J_1, J_2\subset D_k$, where $k\leq O_r(1)$ such that 
$$\bigcup \limits_{g\in J_2}gD_{Q,2}\supset D\subset \bigcup\limits_{g\in J_1}D_{Q,1}g$$
and $|D_{Q,1}|, |D_{Q,2}| \gg_r |D|$ and $|J_1|, |J_2|\leq |F:Q|\ll_r 1$. 
Let $E=D_{Q,1}\cup D_{Q,2}$. Then \eqref{e: d grows} implies that
$|EEE|\leq C^{O_r(1)}|E|$; by Lem.~\ref{l: connect}, this means that $E_3$ is a $C^{O_r(1)}$-approximate group.

We apply Cor.~\ref{c: tointon} to conclude that $E_3$ is
$\exp(C^{O_r(1)})$-controlled by a coset nilprogression $HP$ of rank
$C^{O_r(1)}$ contained in $(E_3)^{C^{O_r(1)}}$. In other words,
there is a set $X$ with $|X|\leq \exp(C^{O_r(1)})$ 
such that $E_3\subset HPX\cap XHP$. Then $D\subset H P X J_2 \cap J_1 X H
P$. Since \[|J_1 X \cup X J_2|\leq |X| |J_2| + |J_1| |X| \leq
 O_r(1) \cdot \exp\left(C^{O_r(1)}\right) = \exp\left(C^{O_r(1)}\right).\]
and so $D$ is $\exp(C^{O_r(1)})$-controlled by a coset nilprogression $HP$
of rank $C^{O_r(1)}$.

The preimage of $HP$ in $H_2$, $\pi^{-1}(HP) = H_1HP$, is a coset nilprogression of rank $C^{O_r(1)}$ that $\exp(C^{O_r(1)})$-controls the set $A'$. What is more, by definition, $H_1HP$ is contained in $A^{C^{O_r(1)}}$. Let $W$ be a set of cardinality $\exp(C^{O_r(1)})$ such that $A' \subseteq WH_1HP \cap H_1HPW$.

Define $B=A_k\cap H_2$ and, appealing to the Tripling Lemma, observe that
$$|AB|, |BA|\leq |A_{k+1}| \leq C^{O_r(1)}|A| \leq C^{O_r(1)}|B|.$$
Then Lem.~\ref{l: ruzsa cover} implies that there exist sets $Y_1, Y_2$, both of cardinality $C^{O_r(1)}$, such that $Y_2BB^{-1} \supseteq A \subseteq B^{-1}B Y_1$; in particular $A \subseteq Y_2A'\cap A'Y_1$. We may assume that $1\in Y_1\cap Y_2$. We conclude that
$$A \subseteq Y_2WY_1H_1HP \cap H_1HPY_2WY_1.$$
In other words, $A$ is $\exp(C^{O_r(1)})$-controlled by $H_1 H P$, as required.
\end{proof}

\providecommand{\bysame}{\leavevmode\hbox to3em{\hrulefill}\thinspace}
\providecommand{\MR}{\relax\ifhmode\unskip\space\fi MR }
\providecommand{\MRhref}[2]{%
  \href{http://www.ams.org/mathscinet-getitem?mr=#1}{#2}
}
\providecommand{\href}[2]{#2}


\begin{thebibliography}{Hum75}


\bibitem[BG08]{bgsu2}
J.~Bourgain and A.~Gamburd, \emph{On the spectral gap for finitely-generated
  subgroups of $\rm {SU}(2)$}, Invent. Math. \textbf{171} (2008), no.~1,
  83--121.

\bibitem[BG08b]{bgexp}
J.~Bourgain and A.~Gamburd, \emph{Uniform expansion bounds for {C}ayley graphs of {${\rm SL}_2(\mathbb{F}_p)$}}, Ann. of Math. (2) \textbf{167} (2008), no.~2, 625--642.

\bibitem[BKT04]{bkt}
J.~Bourgain, N.~Katz, N. and T.~Tao, \emph{A sum-product estimate in finite fields, and applications}, Geom. Funct. Anal. \textbf{14} (2004), no.~1, 27--57.

\bibitem[BK03]{bk}
J.~Bourgain and S.~V.Konyagin, \emph{Estimates for the number of sums and products and for exponential sums over subgroups in fields of prime order}, C. R. Math. Acad. Sci. Paris, \textbf{337} (2003), no.~2, 75--80.

\bibitem[BG11a]{BG1}
E.~Breuillard and B.~Green, \emph{Approximate groups, {I}: the torsion-free
  nilpotent case},  J. Inst. Math. Jussieu \textbf{10} (2011), no. 1, 37--57.

\bibitem[BG11b]{BG2}
\bysame, \emph{Approximate groups, {II}: the solvable linear case},  Q. J. Math. \textbf{62} (2011), no. 3, 513--521.

\bibitem[BGT11]{bgt2}
E.~Breuillard, B.~Green, and T.~Tao, \emph{Approximate subgroups of linear
  groups},  Geom. Funct. Anal. \textbf{21} (2011), no. 4, 774--819. 

\bibitem[BGT12]{bgt11}
E.~Breuillard, B.~Green, and T.~Tao, \emph{The structure of approximate groups}, 
Publ. Math. Inst. Hautes \'Etudes Sci. \textbf{116} (2012), Issue 1, 115--221. 

\bibitem[Bor91]{borel}
A.~Borel, \emph{Linear algebraic groups}, second ed., Graduate Texts in
  Mathematics, vol. 126, Springer-Verlag, New York, 1991.

\bibitem[BS68]{boreltits2}
A.~Borel and T.~A. Springer, \emph{Rationality properties of linear algebraic
  groups. {II}}, T\^ohoku Math. J. (2) \textbf{20} (1968), 443--497.

\bibitem[BT71]{boreltits}
A.~Borel and J.~Tits, \emph{\'{E}l\'{e}ments unipotents et sous-groupes
  paraboliques de groupes r\'{e}ductifs. {I}}, Invent. Math. \textbf{12}
  (1971), 95--104.

\bibitem[Car93]{carter}
R.~W. Carter, \emph{Finite groups of Lie type. Conjugacy classes and complex characters}, Reprint of the 1985 original. Wiley Classics Library. John Wiley and Sons, Ltd., Chichester, 1993.

\bibitem[Cha08]{chang}
M.-C. Chang, \emph{Product theorems in {${\rm SL}_2$} and {${\rm SL}_3$}} J. Inst. Math. Jussieu, \textbf{7} (2008), no~1, 1--25.


\bibitem[Dan94]{danilov}
V.~I. Danilov, \emph{Algebraic varieties and schemes}, Algebraic geometry, {I},
  Encyclopaedia Math. Sci., vol.~23, Springer, Berlin, 1994, pp.~167--297.

\bibitem[Din11]{dinai}
O.~Dinai, \emph{Growth in ${SL}_2$ over finite fields},  J. Group Theory \textbf{14} (2011), no. 2, 273--297.

\bibitem[FKP10]{FKP}
D.~Fisher, N.~Katz, and I.~Peng, \emph{Approximate multiplicative groups in
  nilpotent {L}ie groups}, Proc. Amer. Math. Soc. \textbf{138} (2010),
  1575--1580.

\bibitem[GH11]{helfgill}
N.~Gill and H.~A. Helfgott, \emph{Growth of small generating subsets in
  ${SL}_n(\mathbb{Z}/p \mathbb{Z})$},  Int. Math. Res. Not. IMRN (2011), no. 18, 4226--4251.

\bibitem[GK07]{gk}
A.~A. Glibichuk and S.~V. Konyagin, \emph{Additive properties of product sets
  in fields of prime order}, Additive combinatorics, CRM Proc. Lecture Notes,
  vol.~43, Amer. Math. Soc., Providence, RI, 2007, pp.~279--286.

\bibitem[GLS98]{gls3}
D.~Gorenstein and R.~Lyons and R.~Solomon, \emph{The classification of the finite simple groups. Number 3. Part I. Chapter A. Almost simple K-groups}, Mathematical Surveys and Monographs, 40.3. American Mathematical Society, Providence, RI, 1998.

\bibitem[Gre05]{green}
B.~Green, \emph{Finite field models in additive combinatorics}, Surveys in combinatorics 2005, London Mathematical Sociey Lecture Note Series, 327, Cambridge University Press, Cambridge, 2005, pp.~1--27. 

\bibitem[Hel08]{helfgott2}
H.~A. Helfgott, \emph{Growth and generation in {${\rm SL}\sb 2(\mathbb{Z}/p
  \mathbb{Z})$}}, Ann. of Math. (2) \textbf{167} (2008), no.~2, 601--623.

\bibitem[Hel11]{helfgott3}
\bysame, \emph{Growth and generation in ${SL}_3(\mathbb{Z}/p \mathbb{Z})$}, J.
  Eur. Math. Soc. (JEMS) \textbf{13} (2011), no.~3, 761--851.

\bibitem[Hru12]{hrush}
E.~Hrushovski, \emph{Stable group theory and approximate subgroups}, J. Amer. Math. Soc. \textbf{25} (2012), no.~1, 189--243.

\bibitem[Hum75]{humphreys3}
J.~E. Humphreys, \emph{Linear algebraic groups}, Springer-Verlag, New York,
  1975, Graduate Texts in Mathematics, No. 21.

\bibitem[Kir08]{kirillov}
A.~Kirillov, Jr., \emph{An introduction to {L}ie groups and {L}ie algebras},
  Cambridge Studies in Advanced Mathematics, vol. 113, Cambridge University
  Press, Cambridge, 2008.

\bibitem[KL90]{kl}
P.~Kleidman and M.~Liebeck, \emph{The subgroup structure of the finite
  classical groups}, London Mathematical Society Lecture Note Series, vol. 129,
  Cambridge University Press, Cambridge, 1990.

\bibitem[LR04]{lr}
J.~C. Lennox and D.~J. S. Robinson, \emph{The theory of infinite soluble groups},
Oxford Mathematical Monographs, The Clarendon Press, Oxford University Press, Oxford, 2004.

\bibitem[Mal51]{malcev}
A.~I. Malʹcev, \emph{On some classes of infinite soluble groups}, Mat. Sbornik N.S. \textbf{28(70)} (1951), 567--588.

\bibitem[McN02]{mcninch}
G.~J. McNinch, \emph{Abelian unipotent subgroups of reductive groups}, J.
  Pure Appl. Algebra \textbf{167} (2002), no.~2-3, 269--300.

\bibitem[Ols84]{olson}
J.~E. Olson, \emph{On the sum of two sets in a group}, J. Number Theory
  \textbf{18} (1984), no.~1, 110--120.

\bibitem[PS]{ps2}
L.~Pyber and E.~Szab\'o, \emph{Growth in finite simple groups of lie type of
  bounded rank}, 2010, Preprint available on the Math arXiv: {\tt
  http://arxiv.org/abs/1005.1858}.

\bibitem[Rob82]{robinson}
D.~J.~S. Robinson, \emph{A course in the theory of groups}, Graduate Texts in
  Mathematics, vol.~80, Springer-Verlag, New York, 1982.

\bibitem[RT85]{ruztur}
I.~Z. Ruzsa and S.~Turj{\'a}nyi, \emph{A note on additive bases of integers},
  Publ. Math. Debrecen \textbf{32} (1985), no.~1-2, 101--104.

\bibitem[San12]{sanders}
T.~Sanders, \emph{Approximate groups and doubling metrics}, Math. Proc. Cambridge Philos. Soc. \textbf{152} (2012), no.~3, 385--404.

\bibitem[Ser03]{seress}
A.~Seress, \emph{Permutation group algorithms}, Cambridge Tracts in
  Mathematics, vol. 152, Cambridge University Press, Cambridge, 2003.

\bibitem[Spr09]{springer}
T.~A. Springer, \emph{Linear algebraic groups}, second ed., Modern Birkh\"auser
  Classics, Birkh\"auser Boston Inc., Boston, MA, 2009.

\bibitem[Tao]{taoblog}
T.~Tao, See blog post and subsequent discussion at \\ {\tt
  http://terrytao.wordpress.com/2009/06/21/freimans-theorem-for-solvable-group%
s/}.

\bibitem[Tao08]{taononcomm}
\bysame, \emph{Product set estimates for non-commutative groups}, Combinatorica
  \textbf{28} (2008), no.~5, 547--594.

\bibitem[Tao10]{taofrei}
\bysame, \emph{Freiman's theorem for solvable groups}, Contrib. Discrete Math. \textbf{5} (2010), no.~2, 137--184.
  
\bibitem[TV06]{taovu}
T.~Tao and V.~Vu, \emph{Additive combinatorics}, Cambridge Studies in Advanced
  Mathematics, vol. 105, Cambridge University Press, Cambridge, 2006.

\bibitem[Toi]{tointon}
M.~Tointon, \emph{Freiman's theorem in an arbitrary nilpotent group}, 2010, Preprint available on the Math arXiv: {\tt http://arxiv.org/abs/1211.3989}.
  
\bibitem[Var12]{varju}
P.~Varju, \emph{Expansion in {$\SL_d(\mathcal{O}_K/I)$}, {$I$} square-free}, J. Eur. Math. Soc. (JEMS) \textbf{14} (2012), no.~1, 273--305.

\end{thebibliography}
\end{document}